\documentclass[a4paper]{amsart}
\usepackage{amsmath,amssymb}
\usepackage{stmaryrd}
\usepackage{tikz}
\usetikzlibrary{arrows}
\usepackage{hyperref}

\usepackage{color}

\title[Weihrauch goes Brouwerian]
{Weihrauch goes Brouwerian}

\author[V.\ Brattka]{Vasco Brattka}
\address{Faculty of Computer Science, Universit\"at der Bundeswehr M\"unchen, Germany and 
             Department of Mathematics \& Applied Mathematics, University of Cape Town, South Africa\footnote{Vasco Brattka has received funding from the National Research Foundation of South Africa}} 
\email{Vasco.Brattka@cca-net.de}

\author[G.\ Gherardi]{Guido Gherardi}
\address{Dipartimento di Filosofia e Comunicazione, Universit\`{a} di Bologna, Italy}
\email{Guido.Gherardi@unibo.it}

\def\CC{{\mathcal C}}

\def\LL{{\mathcal L}}
\def\MM{{\mathcal M}}

\def\WW{{\mathcal W}}
\def\XX{{\mathcal X}}

\def\IN{{\mathbb{N}}}

\def\IR{{\mathbb{R}}}

\def\TO{\Longrightarrow}
\def\In{\subseteq}

\def\prefix{\sqsubseteq}

\def\mto{\rightrightarrows}

\def\id{{\rm id}}

\def\dom{{\rm dom}}
\def\range{{\rm range}}

\def\Form{{\rm Form}}

\def\t{\mathrm t}
\def\r{\mathrm r}

\newcommand{\SO}[1]{{{\bf\Sigma}^0_{#1}}}

\def\LPO{\text{\rm\sffamily LPO}}
\def\LLPO{\text{\rm\sffamily LLPO}}
\def\WKL{\text{\rm\sffamily WKL}}

\def\IVT{\text{\rm\sffamily IVT}}

\def\BWT{\text{\rm\sffamily BWT}}

\def\C{\mbox{\rm\sffamily C}}

\def\LPO{\mbox{\rm\sffamily LPO}}
\def\LLPO{\mbox{\rm\sffamily LLPO}}

\def\WBWT{\text{\rm\sffamily WBWT}}

\def\K{\text{\rm\sffamily K}}

\def\KL{\text{\rm\sffamily KL}}

\def\J{\text{\rm\sffamily J}}

\def\NON{\text{\rm\sffamily NON}}
\def\DNC{\text{\rm\sffamily DNC}}
\def\ACC{\text{\rm\sffamily ACC}}

\def\MLR{\text{\rm\sffamily MLR}}
\def\WWKL{\text{\rm\sffamily WWKL}}
\def\GEN{\text{\rm\sffamily GEN}}

\def\COH{\text{\rm\sffamily COH}}
\def\PA{\text{\rm\sffamily PA}}

\def\SORT{\text{\rm\sffamily SORT}}

\def\leqM{\mathop{\leq_{\mathrm{M}}}}

\def\leqT{\mathop{\leq_{\mathrm{T}}}}

\def\equivT{\mathop{\equiv_{\mathrm{T}}}}

\def\leqW{\mathop{\leq_{\mathrm{W}}}}
\def\leqTW{\mathop{\leq_{\mathrm{tW}}}}
\def\equivW{\mathop{\equiv_{\mathrm{W}}}}
\def\equivTW{\mathop{\equiv_{\mathrm{tW}}}}
\def\nequivW{\mathop{\not\equiv_{\mathrm{W}}}}

\def\leqSW{\mathop{\leq_{\mathrm{sW}}}}
\def\leqSTW{\mathop{\leq_{\mathrm{stW}}}}

\def\equivSW{\mathop{\equiv_{\mathrm{sW}}}}
\def\equivSTW{\mathop{\equiv_{\mathrm{stW}}}}

\def\nleqW{\mathop{\not\leq_{\mathrm{W}}}}
\def\nleqTW{\mathop{\not\leq_{\mathrm{tW}}}}
\def\nleqSW{\mathop{\not\leq_{\mathrm{sW}}}}
\def\nleqSTW{\mathop{\not\leq_{\mathrm{stW}}}}

\def\lW{\mathop{<_{\mathrm{W}}}}

\newcommand{\RT}[2]{\text{\rm\sffamily RT}^{#1}_{#2}}

\def\bigtimes{\mathop{\mathsf{X}}}

\newcommand{\dash}{\mbox{-}}

\date{\today}

\newtheorem{theorem}{Theorem}[section]
\newtheorem{proposition}[theorem]{Proposition}
\newtheorem{lemma}[theorem]{Lemma}
\newtheorem{fact}[theorem]{Fact}

\newtheorem{corollary}[theorem]{Corollary}

\theoremstyle{definition}
\newtheorem{definition}[theorem]{Definition}

\newtheorem{example}[theorem]{Example}

\begin{document}

\begin{abstract}
We prove that the Weihrauch lattice can be transformed into a Brouwer algebra by the consecutive application 
of two closure operators in the appropriate order: first completion and then parallelization.
The closure operator of completion is a new closure operator that we introduce. It transforms
any problem into a total problem on the completion of the respective types, where we allow any
value outside of the original domain of the problem. 
This closure operator is of interest by itself, as it generates a total version of Weihrauch reducibility
that is defined like the usual version of Weihrauch reducibility, but in terms of total realizers.
From a logical perspective completion can be seen as a way to make problems independent of their premises.
Alongside with the completion operator and total Weihrauch reducibility we need to study precomplete representations
that are required to describe these concepts. 
In order to show that the parallelized total Weihrauch lattice forms a Brouwer algebra, 
we introduce a new multiplicative version of an implication.
While the parallelized total Weihrauch lattice forms a Brouwer algebra with this implication,
the total Weihrauch lattice fails to be a model of intuitionistic linear logic in two different ways. In order to pinpoint the algebraic reasons
for this failure, we introduce the concept of a Weihrauch algebra that allows us to formulate
the failure in precise and neat terms. Finally, we show that the Medvedev Brouwer algebra
can be embedded into our Brouwer algebra, which also implies that the theory of our Brouwer algebra is Jankov logic.
\end{abstract}

\maketitle

%\begin{footnotesize}
\setcounter{tocdepth}{1}
\tableofcontents
%\end{footnotesize}

\pagebreak

\section{Introduction}
\label{sec:introduction}

Over the previous ten years Weihrauch complexity has been developed as a 
computability theoretic approach to classify the uniform computational content of theorems.
A survey article that summarizes some of the current research directions in Weihrauch complexity
can be found in \cite{BGP18}.\footnote{A comprehensive up-to-date bibliography
is maintained at the following web page: {\tt http://cca-net.de/publications/weibib.php}}
The advantage of this approach is that it provides a direct computability theoretic way
to classify problems, while heuristic observation shows that the approach can
be seen as a uniform version of reverse mathematics in the sense of Friedman and Simpson~\cite{Sim09}.

Weihrauch complexity is based on Weihrauch reducibility $\leqW$ that induces a 
lattice structure. Beyond the lattice operations the Weihrauch lattice is equipped with a number of additional algebraic operations. 
Early on it was noticed that the semantics of these operations has the flavor of 
linear logic. Table~\ref{tab:linear-logic} provides
a dictionary that shows how the usual symbols for operations on problems in the Weihrauch
lattice are translated into operations of linear logic.

\begin{table}[htb]
\begin{center}
\begin{tabular}{ll}
{logical operation in linear logic\ } & {algebraic operation on problems}\\\hline
$\otimes$ multiplicative conjunction  & $\times$ product \\
$\&$ additive conjunction                & $\sqcup$ coproduct\\
$\oplus$ additive disjunction           & $\sqcap$ infimum\\
\rotatebox[origin=c]{180}{$\&$} multiplicative disjunction         & $+$ sum\\
$!$ bang		                                    & $\widehat{\ }$ parallelization
\end{tabular}
\end{center}
\caption{Linear logic versus the algebra of problems}
\label{tab:linear-logic}
\end{table}

However, so far no satisfactory interpretation of the Weihrauch lattice as
a model of (intuitionistic) linear logic has been found.
This is partially due to the lack of an internal implication operation
that corresponds to the linear implication $\multimap$.
Such an implication would have to fulfill 
\[(g\multimap f)\leqW h\iff f\leqW g\times h\]
and it can be proved that such an implication does not exist, given $\leqW$ and $\times$~\cite[Proposition~37]{BP18}.
However, Weihrauch reducibility $f\leqW g$ can be seen at least as an
external implication operation $f\Longleftarrow g$.
 
The Weihrauch lattice has also additional algebraic operations such 
as the compositional product $\star$, which can be seen as a non-commutative
version of conjunction. Here $f\star g$ captures what can be computed by
first using the problem $g$ and then the problem $f$, possibly with some
intermediate computation. There is an implication operation $g\to f$ in the Weihrauch 
lattice that is a right co-residual operation of $\star$ \cite{BP18}, i.e., we have
\[(g\to f)\leqW h\iff f\leqW g\star h.\]
However, this setting does not provide a model for classical linear logic, 
since the operation $\star$ is not commutative.\footnote{Girard also proposed a less known non-commutative 
version of linear logic, but also this logic does not seem to fit to our model~\cite{Yet90}.}

While the connections to linear logic might not be as tight as one wishes,
there is still hope that there is a close connection to intuitionistic logic.
In linear logic intuitionistic implication is represented by $!A\multimap B$.
Hence, it is to be expected that the parallelized Weihrauch reducibility $f\leqW\widehat{g}$ gives us an
external form of intuitionistic implication. This could theoretically be substantiated
by showing that the resulting structure is a Brouwer algebra, since Brouwer algebras are models
for intermediate propositional logics in between classical and intuitionistic logic.
However, also this hope did not materialize as Higuchi and Pauly proved
that the parallelized Weihrauch lattice is not a Brouwer algebra~\cite{HP13}.

In this article we prove that one does obtain a Brouwer algebra if one
combines two closure operators in the Weihrauch lattice in the appropriate
order: first completion $f\mapsto\overline{f}$ and then parallelization $f\mapsto\widehat{f}$.
While parallelization is a well understood operation \cite{BG11} that corresponds
somewhat to the usage of countable choice in constructive mathematics,
completion is a new operation that we introduce in this article.
Formally, the completion $\overline{f}:\overline{X}\mto\overline{Y}$
of a problem $f:\In X\mto Y$ is defined by
\[\overline{f}(x):=\left\{\begin{array}{ll}
   f(x) & \mbox{if $x\in\dom(f)$}\\
   \overline{Y} & \mbox{otherwise}
\end{array}\right.,\]
i.e., by a totalization of $f$ on the completions $\overline{X},\overline{Y}$ of the corresponding types.\footnote{We were inspired 
to continue the study of completions by recent work of Dzhafarov who used them to show that strong Weihrauch reducibility
induces a lattice structure~\cite{Dzh19}.}

Logically, completion can be seen as a way to make problems independent
of their premises. In general, a logical statement of the form
\[(\forall x\in X)(x\in D\TO(\exists y\in Y)\;P(x,y))\]
is translated into a problem $f:\In X\mto Y$ in the Weihrauch lattice
by setting $\dom(f)=D$ and $f(x):=\{y\in Y:P(x,y)\}$ for all $x\in\dom(f)$.
Now the transition to the completion $\overline{f}$ corresponds to
the statement
\[(\forall x\in\overline{X})(\exists y\in\overline{Y})(x\in D\TO P(x,y)),\]
where the existence is required independent of the premise $x\in D$.
The completion of the data types is relevant here, as it guarantees the existence of total
representations of the underlying types.

The completion operation $f\mapsto\overline{f}$ is of interest by itself as it is a closure
operator that yields a total version of Weihrauch reducibility $\leqTW$ by $f\leqTW g\iff f\leqW\overline{g}$.
Total Weihrauch reducibility $\leqTW$ can also be
defined directly almost as the usual reducibility $\leqW$, but in terms of total realizers instead of partial realizers.
In this case the completion of the types features again,
since one needs to consider so-called precomplete representations
for the underlying types. 

Among other things we prove that total Weihrauch reducibility induces
a lattice structure with operations induced by the original operations of the Weihrauch lattice.
The lattice structure of the total Weihrauch lattice is somewhat different from the original Weihrauch lattice,
but it does not change all too dramatically as many problems are actually complete, 
i.e., Weihrauch equivalent to their own completion.
We list some examples of complete and incomplete problems:

\begin{itemize}
\item Complete problems: $\LPO,\LLPO,\lim,\J,\WKL,\SORT,\IVT,\PA,\MLR,\DNC_n$.
\item Incomplete problems: $\C_\IN,\C_{\IN^\IN},\WWKL$.
\end{itemize}

The reader who does not know these problems will find relevant definitions of some of them
later. The topic of completion of choice problems is subject of an entirely separate article~\cite{BG19}.

When we move to the total Weihrauch lattice $\WW_{\rm tW}$ of total Weihrauch reducibility $\leqTW$,
then we can introduce a new implication $f\twoheadrightarrow g$ that
can almost be seen as a multiplicative co-residual of $\times$. 
However, also in this case we fail to obtain a model for intuitionistic linear logic.

\begin{figure}[htb]
\begin{center}
\begin{tikzpicture}
\footnotesize
\draw[rotate=30]  (0.2098,2.7634) ellipse (1.5 and 3);
\draw[rotate=-30]  (-0.5502,2.5634) ellipse (1.5 and 3);
\draw  (-0.2,3.2) ellipse (4 and 3.5);
\node at (-0.2,5.8) {Weihrauch algebras};
\node at (-2,4.4) {commutative};
\node at (1.6,4.4) {deductive};
\draw  (-0.2,0.8) ellipse (0.9 and 0.9);
\node at (-0.2,0.8) {\begin{minipage}{1.3cm}\begin{center}Brouwer\\algebras\end{center}\end{minipage}};
\node at (-0.2,2.2) {\begin{minipage}{1.3cm}\begin{center}Troelstra\\algebras\end{center}\end{minipage}};
\node at (-2,3.8) {$\times,\twoheadrightarrow$};
\node at (1.6,3.8) {$\star,\to$};
\node at (5.2,2.6) {\begin{minipage}{2cm}\begin{center}models of inuitionistic linear logics\end{center}\end{minipage}};
\node at (5.2,0.8) {\begin{minipage}{2cm}\begin{center}models of inuitionistic logics\end{center}\end{minipage}};

\node (ILL) at (4.2,2.6) {};
\node (IL) at (4.2,0.8) {};
\node (T) at (0.35,2.2) {};
\node (B) at (0.4,0.8) {};
\draw [->,looseness=1] (ILL) to [out=170,in=30] (T);
\draw [->,looseness=1] (IL) to [out=190,in=-30] (B);

\end{tikzpicture}
\end{center}
\ \\[-0.3cm]
\caption{Different types of algebras as models of logic}
\label{fig:algebras}
\end{figure}
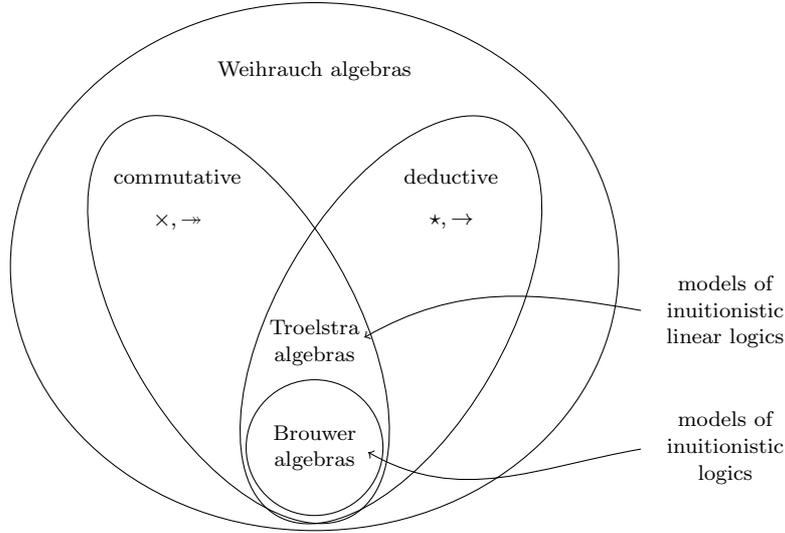

In order to make the spectacular twofold failure of obtaining a model
of intuitionistic linear logic more understandable, we introduce the concept
of a Weihrauch algebra in the following section~\ref{sec:order}.
These are lattice-ordered monoids with some additional implication operation.
The total Weihrauch lattice $\WW_{\rm tW}$ is a commutative Weihrauch algebra
with respect to $\times,\twoheadrightarrow$ and a deductive Weihrauch algebra with respect to $\star,\to$.
However, none of these Weihrauch algebras is commutative and deductive simultaneously,
which is what is required in order to obtain, in our terms, a Troelstra algebra\footnote{This is the dual 
structure of what Troelstra called an {\em intuitionistic linear algebra}~\cite{Tro92}.}, i.e., a model
of some form of intuitionistic linear logic. See the diagram in Figure~\ref{fig:algebras} for an illustration
of the situation.

When we apply parallelization after completion, then we obtain the parallelized
total Weihrauch lattice $\WW_{\rm ptW}$ which then leads to a Brouwer algebra,
i.e., a Troelstra algebra where the monoid structure is merged with the
lattice structure (in our terms $\times$ and $\sqcup$ are merged).
In section~\ref{sec:Brouwer} we prove that one can embed the Medvedev Brouwer algebra~\cite{Sor96} into
our Brouwer algebra. Like in the case of the Medvedev Brouwer algebra we
obtain Jankov logic as the theory of our algebra.

In the following section~\ref{sec:order} we provide some very basic lattice theoretic
results regarding closure operators that are helpful for our study, and we define
Weihrauch and Troelstra algebras alongside with Brouwer algebras.
In section~\ref{sec:precomplete} we study precomplete representations
and the data type of completion that is needed to introduce the
closure operator of completion and total Weihrauch reducibility.
In section~\ref{sec:total} we introduce total Weihrauch reducibility
and we prove some basic properties of it.
In section~\ref{sec:completion} we introduce and study the closure operator
of completion. Section~\ref{sec:algebra} provides results that show
how the algebraic operations of the Weihrauch lattice interact with completion.
In particular, we prove that total Weihrauch reducibility actually yields a 
lattice structure. 
In section~\ref{sec:residual} we review the operations $\star$ and $\to$
and study their interaction with completion and we also introduce
the new implication operation $\twoheadrightarrow$.
Finally, in section~\ref{sec:Brouwer} we prove that the 
parallelized total Weihrauch lattice $\WW_{\rm ptW}$ is a Brouwer
algebra with the implication derived from $\twoheadrightarrow$.
We also discuss the embedding of the Medvedev lattice.
We close this article with a brief survey on the classification
of concrete problems in the parallelized total Weihrauch lattice.

\section{Closure Operators and Weihrauch Algebras}
\label{sec:order}

In this section we prepare some basic order theoretic concepts that we are going to use frequently.
We recall that a {\em preorder} $\leq$ on a set $X$ is a binary relation on $X$ that is reflexive and transitive.
We also speak of a {\em preordered space} $(X,\leq)$ in this context.
An {\em equivalence relation} $\equiv$ on a set $X$ is a binary relation on $X$ that is reflexive, symmetric and transitive.
In the following we will have to deal with several closure operators.

\begin{definition}[Closure operator]
\label{def:closure-operator}
Let $(X,\leq)$ be a preordered space together with a map $c:X\to X$. 
Then $c$ is called a {\em closure operator}, if
\begin{enumerate}
\item $x\leq c(x)$, 
\item $cc(x)\leq c(x)$ and 
\item $x\leq y\TO c(x)\leq c(y)$ 
\end{enumerate}
hold for all $x,y\in X$.
We say that $x\in X$ is {\em closed} if $c(x)\leq x$. 
\end{definition}

We call a map $c:X\to X$ {\em monotone}, if $x\leq y\TO c(x)\leq c(y)$ holds
and {\em antitone}, if $x\leq y\TO c(y)\leq c(x)$ holds. 
We use the same terminology for binary maps $\Box:X\times X\to X$ with respect to individual arguments.
We use the usual concepts of a {\em suprema} (also called a least upper bound) and an {\em infima} (also called a greatest lower bound) for
preordered sets in the usual way, and we note that on a preordered space they are only uniquely determined up to equivalence in the case of existence.
If one has a preordered space $(X,\leq)$ and one identifies all equivalent elements with each other, then one obtains a quotient structure $(X/\!\!\equiv,\leq)$,
which is a {\em partially ordered space}, i.e., the resulting order is a preorder that is additionally anti-symmetric. 
A {\em lattice} $(X,\leq,\wedge,\vee)$ is a partially ordered set together with a supremum operation $\vee$ and an infimum operation $\wedge$.
If $\leq_c$ is a preorder on $X$ and  $c:X\to X$ a map, then we say that $c$ {\em generates} $\leq_c$ on $(X,\leq)$ if 
$x\leq_c y\iff x\leq c(y)$ holds for all $x,y\in X$.
The following result is straightforward to prove. It shows how closure operators act on lattices
and preordered spaces.

\begin{proposition}[Closure operators]
\label{prop:closure-operators}
Let $(X,\leq)$ be a preordered space with two closure operators $c,c':X\to X$
and binary operations $\Box,\vee,\wedge:X\times X\to X$. Then
\begin{enumerate}
\item $x\leq_c y:\iff x\leq c(y)\iff c(x)\leq c(y)$ defines a preorder that satisfies $x\leq y\TO x\leq_c y$ for all $x,y\in X$,
\item $x\equiv_c y:\iff (x\leq_c y$ and $y\leq_c x)$ defines an equivalence relation,
\item $\Box_c:X\times X\to X,(x,y)\mapsto c(x)\Box c(y)$ shares corresponding monotonicity properties as $\Box$, more precisely:
\begin{enumerate}
\item if $\Box$ is monotone (antitone) in one argument with respect to $\leq$, then so is $\Box_c$ in the same argument with respect to $\leq_c$,
\item if $\wedge$ is an infimum with respect to $\leq$, then so is $\wedge_c$ with respect to $\leq_c$,
\item if $\vee$ is a supremum with respect to $\leq$, then so is $\vee_c$ with respect to $\leq_c$. 
\end{enumerate}
\item If $(X,\leq,\wedge,\vee)$ is a lattice, then so is $(X/\!\!\equiv_c,\leq_c,\wedge_c,\vee_c)$.
\item $c'\circ c:X\to X$ is monotone with respect to $\leq$ and $\leq_c$.
\end{enumerate}
\end{proposition}
\begin{proof}
(1) Reflexivity of $\leq_c$ follows from $x\leq c(x)$, transitivity from monotonicity of $c$ together with $cc(x)\leq c(x)$.
It is also clear that $x\leq c(y)\iff c(x)\leq c(y)$ holds. 
Finally, $x\leq y\TO x\leq_c y$ holds as $c$ is monotone.\\
(2) Is obvious. \\
(3) (a) Suppose $\Box$ is  antitone in the first argument and $x_1,x_2,y\in X$ with $x_1\leq_c x_2$.
Then $c(x_1)\leq c(x_2)$ and hence
$c(x_2)\Box c(y)\leq c(x_1)\Box c(y)$, since $\Box$ is antitone in the first argument. 
Hence $x_2\Box_c y\leq x_1\Box_c y\leq c(x_1\Box_c y)$, which means $x_2\Box_c y\leq_c x_1\Box_c y$, i.e., $\Box_c$ is antitone
in the first argument with respect to $\leq_c$. The other cases are treated analogously.\\
(b), (c) If $\wedge$ is an infimum with respect to $\leq$ and $x,y\in X$, then $x\wedge y\leq x$ and $x\wedge y\leq y$ and hence $x\wedge y\leq_c x$ and $x\wedge y\leq_c y$.
Hence $x \wedge y$ is a lower bound of $x$ and $y$ with respect $\leq_c$.
Let now $z\in X$ be such that $z\leq_c x$ and $z\leq_c y$. 
Then $z\leq c(x)$ and $z\leq c(y)$, which implies $z\leq c(x)\wedge c(y)$ and hence $z\leq_c x\wedge_c y$. 
This means that $x\wedge_c y$ is above every lower bound
of $x$ and $y$ with respect to $\leq_c$ and hence it is an infimum with respect to $\leq_c$.
The statement for suprema can be proved analogously.\\
(4) This follows from (1)--(3).\\
(5) If $x\leq y$, then $c'c(x)\leq c'c(y)$ follows.
If $x\leq_c y$, then $cc'c(x)\leq cc'c(y)$ follows and hence $c'c(x)\leq_c c'c(y)$.
\end{proof}

We also need to deal with situations where a closure operator respects certain underlying algebraic operations
or other closure operators. Hence, we use the following terminology.

\begin{definition}[Preservation]
\label{def:preservation}
Let $(X,\leq)$ be a preordered space with closure operators $c,c':X\to X$
and a binary operation $\Box:X\times X\to X$. 
\begin{enumerate}
\item We say that $c$ {\em is preserved by} $\Box$ if $c(x\Box y)\leq c(x)\Box c(y)$ for all $x,y\in X$.
\item We say that $c$ {\em is co-preserved by} $\Box$ if $c(x)\Box c(y)\leq c(x\Box y)$ for all $x,y\in X$.
\item We say that $c$ {\em is preserved by} $c'$ if $c\circ c'(x)\leq c'\circ c(x)$ for all $x\in X$.
\end{enumerate}
\end{definition} 

Whenever a closure operator is preserved by a certain operation, then we can draw certain conclusions.
The proof of the following result is straightforward.

\begin{proposition}[Preservation]
\label{prop:preservation}
Let $(X,\leq)$ be a preordered space with closure operators $c,c':X\to X$ and a binary monotone operation $\Box:X\times X\to X$.
\begin{enumerate}
\item If $c$ is preserved by $\Box$, then for all $x,y\in X$
\[c(x\Box y)\leq c(x)\Box c(y)\equiv c(c(x)\Box c(y)).\]
In particular, $x\Box y$ is closed if $x$ and $y$ are.
\item If $c$ is co-preserved by $\Box$, then for all $x,y\in X$
\[c(x)\Box c(y)\leq c(x\Box y)\leq c(c(x)\Box c(y)).\]
\item If $c$ is preserved by $c'$, then for all $x\in X$
\[cc'(x)\leq c'c(x)\equiv cc'c(x).\]
In particular, $c'c$ is a closure operator with respect to $\leq_c$ and $\leq$, and $c'(x)$ is closed with respect to $c$ if $x$ is so.
\end{enumerate}
\end{proposition}
\begin{proof}
The equivalences in (1) and (3) are consequences of the respective first relations and the fact that $c$ is a closure operator.
For the second relation in (2) we just use that $c$ is a closure operator.
It is clear that $c'(x)$ is closed if $x$ is closed. That $c'c$ is monotone with respect to $\leq_c$ follows from Proposition~\ref{prop:closure-operators}.
Clearly, also $x\leq_c c'c(x)$ holds. Finally, $c'cc'c(x)\leq c'c'c(x)\leq c'c(x)$ and hence $c'cc'c(x)\leq_cc'c(x)$. 
\end{proof}

If the binary operation is a supremum or an infimum operation, then it is always
preserved in certain ways.

\begin{proposition}[Preservation of suprema and infima]
\label{prop:preservation-suprema-infima}
Let $(X,\leq)$ be a preordered space with a closure operator $c:X\to X$,
and binary operations $\vee,\wedge:X\times X\to X$.
\begin{enumerate}
\item If $\vee$ is a supremum operation, then it co-preserves $c$.
\item If $\wedge$ is an infimum operation, then it preserves $c$.
\end{enumerate}
In particular, $x\vee y\equiv_c x\vee_c y$, if $\vee$ is a supremum.
\end{proposition}
\begin{proof}
Since $x\vee y$ is a supremum, we obtain $c(x)\leq c(x\vee y)$ and $c(y)\leq c(x\vee y)$ due to monotonicity
of $c$. Hence $c(x)\vee c(y)\leq c(x\vee y)$, which means that $\vee$ co-preserves $c$.
The statement for $\wedge$ can be proved analogously. That $\vee$ co-preserves $c$ means $x\vee_c y\leq_c x\vee y$.
We also have $x\vee y\leq c(x)\vee c(y)$, i.e., $x\vee y\leq_c x\vee_c y$.
\end{proof}

We note that this result implies that the we can replace $\vee_c$ by $\vee$ in Proposition~\ref{prop:closure-operators}.

In the following we will have to deal with lattices that have some additional algebraic operations
and we propose the following concept that encapsulates a structure that we will see in different variations.

\begin{definition}[Weihrauch algebra]
\label{def:Weihrauch-algebra}
We call $(X,\leq,\wedge,\vee,\cdot,\to,1,\bot,\top)$ a {\em Weihrauch algebra} if the following hold:
\begin{enumerate}
\item $(X,\leq,\wedge,\vee)$ is a bounded lattice with bottom $\bot$ and top $\top$.\hfill (Lattice)
\item $(X,\cdot,1)$ is a monoid with neutral element $1$. \hfill (Monoid)
\item $\cdot:X\times X\to X$ is monotone in both components.\hfill (Monotonicity)
\item $\to:X\times X\to X$ is monotone in the second component, antitone in the first component. \hfill (Monotonicity)
\item $x\leq y\cdot z\TO (y\to x)\leq z$ holds for all $x,y,z\in X$. \hfill (Implication)
\end{enumerate}
A Weihrauch algebra is called {\em commutative}, if $\cdot$ is commutative, and it is called
{\em deductive}, if ``$\iff$'' holds instead of ``$\TO$'' in (5).
\end{definition}

One could add additional distributivity requirements to this definition.
Structures that satisfy (1), (2) and  (3) have also been called {\em lattice-ordered monoids}.
Using these building blocks, we can define structures that have been already considered
for other purposes.

\begin{definition}[Algebras]
\label{def:algebras}
Let $\XX=(X,\leq,\wedge,\vee,\cdot,\to,1,\bot,\top)$ be a Weihrauch algebra.
We call $\XX$ a {\em Troelstra algebra} if it is commutative and deductive.
If, additionally, $\cdot=\vee$ and $1=\bot$, then $\XX$ is called a {\em Brouwer algebra}.
\end{definition}

If we denote a Brouwer algebra as a tuple, then we omit the double occurrence of $\cdot=\vee$ and $1=\bot$, respectively.
What we call a Troelstra algebra  is exactly what Troelstra~\cite{Tro92a} called an {\em intuitionistic linear algebra},
except that the order is reversed. 
A bottom element in our sense is not required in Troelstra's axioms, but it always exists by \cite[Lemma~8.3]{Tro92a}. 
The relevance of Troelstra algebras is that they form sound and complete models of intuitionistic linear logic~\cite[Theorem 8.15]{Tro92a}.
In an analogous sense Brouwer algebras (that are just defined dually to Heyting algebras\footnote{The term Brouwer algebra is used in different versions in different references, we mean by a {\em Brouwer algebra} just the dual
concept of a Heyting algebra, as usual in computability theory~\cite{Sor96}.}) are known as models of intermediate logics,
i.e., predicate logics between classical logic and intuitionistic logic~\cite{GJKO07}.

A {\em Brouwer algebra embedding} is an injective map from one Brouwer algebra to another one that 
is monotone in both directions, preserves suprema, infima, implications and the bottom and top elements.

In the case of a deductive Weihrauch algebra the condition (5) can be seen as a law of (co-)residuation.
We need to add the prefix ``co-'' as residuation is normally considered in the opposite order~\cite{GJKO07}. 

\begin{definition}[Co-residuation]
Let $(X,\leq)$ be a preordered set with a binary operation $\cdot:X\times X\to X$.
Then we call $\cdot$ {\em right co-residuated}, if there is a binary operation
$\to:X\times X\to X$ such that
\[x\leq y\cdot z\iff(y\to x)\leq z\]
holds for all $x,y,z\in X$. Analogously, we call $\cdot$ {\em left co-residuated}, if an
analogous condition holds with $z\cdot y$ in place of $y\cdot z$.
Finally, $\cdot$ is called {\em co-residuated} if and only if it is left and right co-residuated.
\end{definition}

Hence, a deductive Weihrauch algebra is a right co-residuated lattice-ordered monoid
and a Troelstra algebra is a co-residuated lattice-ordered monoid.

\section{Precomplete Representations}
\label{sec:precomplete}

We will need some pairing functions in the following. Firstly, we define a pairing function $\pi:\IN^\IN\times\IN^\IN\to\IN^\IN,(p,q)\mapsto\langle p,q\rangle$
by $\langle p,q\rangle(2n):=p(n)$ and $\langle p,q\rangle(2n+1):=q(n)$ for $p,q\in\IN^\IN$ and $n\in\IN$.
We define a pairing function of type $\langle,\rangle:(\IN^\IN)^\IN\to\IN^\IN$ by
$\langle p_0,p_1,p_2,...\rangle\langle n,k\rangle:=p_n(k)$ for all $p_i\in\IN^\IN$ and $n,k\in\IN$,
where $\langle n,k\rangle$ is the standard Cantor pairing defined by $\langle n,k\rangle:=\frac{1}{2}(n+k+1)(n+k)+k$. 
Finally, we note that by $np$ we denote the concatenation of a number $n\in\IN$ with a sequence $p\in\IN^\IN$.
By $\pi_i:\IN^\IN\to\IN^\IN,\langle p_0,p_1,p_2,...\rangle\mapsto p_i$ we denote the projection on the $i$--th
component of a tuple and we also use the binary tupling functions $\pi_1\langle p,q\rangle=p$ and $\pi_2\langle p,q\rangle=q$.
It will always be clear from the context whether we apply these functions in a countable or binary setting.

We recall that a {\em represented space} $(X,\delta)$ is a set $X$ together with
a surjective (partial) map $\delta:\In\IN^\IN\to X$, called the {\em representation} of $X$.
For the purposes of our topic so-called {\em precomplete representations}
are important. They were introduced by Kreitz and Weihrauch \cite{KW85} following
the concept of a precomplete numbering, that was originally introduced by Er{\v{s}}ov~\cite{Ers99}.

\begin{definition}[Precompleteness]\rm
A representation $\delta:\In\IN^\IN\to X$ is called {\em precomplete},
if for any computable function $F:\In\IN^\IN\to\IN^\IN$ there exists
a total computable function $G:\IN^\IN\to\IN^\IN$ such that
$\delta F(p)=\delta G(p)$ for all $p\in\dom(F)$.
\end{definition}

In this situation we also say that the represented space $(X,\delta)$ is {\em precomplete}.
We point out that we demand that the equation in the definition holds for all $p\in\dom(F)$, not
only for $p\in\dom(\delta F)$.
The precomplete representations are exactly those that satisfy
a certain version of the recursion theorem~\cite{KW85}.
For us they are relevant since we are going to work with
total functions. It is clear that not all represented spaces are precomplete.
By ${\id:\IN^\IN\to\IN^\IN}$ we denote the {\em identity} of Baire space. 
For other sets $X$ we usually add an index $X$ and write the {\em identity} as $\id_X:X\to X$.
By $\widehat{n}:=nnn...\in\IN^\IN$ we denote the constant sequence with value $n\in\IN$.

\begin{example}
\label{ex:precompleteness}
There are partial computable functions $F:\In\IN^\IN\to\IN^\IN$ without total computable
extension, such as the function defined by $F(p)=\widehat{n}:\iff p$ starts with exactly $n$ digits $0$,
where $\dom(F)=\{0^np:n\in\IN,p\in\IN^\IN,p(0)\not=0\}$.
This shows that the represented space $(\IN^\IN,\id)$ is not precomplete.
\end{example}

However, it is not too hard to see that in every equivalence class of representations there is a precomplete representation.\footnote{This result is due to Matthias Schr\"oder (personal communication 2009), see the construction
in \cite[Lemmas~4.2.10, 4.2.11, Section~4.2.5]{Sch02c}.}
We recall that for two representations $\delta_1,\delta_2$ of the same set $X$
we say that $\delta_1$ is {\em computably reducible} to $\delta_2$, in symbols $\delta_1\leq\delta_2$,
if and only if there is a computable $F:\In\IN^\IN\to\IN^\IN$ such that $\delta_1=\delta_2F$.
We denote the corresponding equivalence by $\equiv$. 
For $p\in\IN^\IN$ we denote by $p-1\in\IN^\IN\cup\IN^*$ the sequence or word that is formed as concatenation of 
$p(0)-1$, $p(1)-1$, $p(2)-1$,... with the understanding that $-1=\varepsilon$ is the empty word.

\begin{definition}[Precompletion]
\label{def:precompletion}
Let $(X,\delta_X)$ be a represented space. Then the {\em precompletion} $\delta_X^\wp$ of
$\delta_X$ is defined by $\delta_X^\wp(p):=\delta_X(p-1)$ for all $p\in\IN^\IN$ such that
$p-1\in\dom(\delta_X)$.
\end{definition}

We note that the identity $\id:\IN^\IN\to\IN^\IN$, considered as a representation of $\IN^\IN$,
has the precompletion $\id^\wp:\In\IN^\IN\to\IN^\IN$ with $\id^\wp(p):=p-1$ and
in general $\delta_X^\wp=\delta_X\circ\id^\wp$.
Now we can prove the following result.

\begin{proposition}[Precompleteness]
\label{prop:precompleteness}
Let $(X,\delta_X)$ be a represented space. The precompletion $\delta_X^\wp$ of $\delta_X$
is precomplete and satisfies $\delta_X^\wp\equiv\delta_X$.
\end{proposition}
\begin{proof}
The computable function $F:\IN^\IN\to\IN^\IN,p\mapsto p+1$ satisfies $\delta_X(p)=\delta_X^\wp F(p)$ and hence it witnesses $\delta_X\leq\delta_X^\wp$.
The computable function $G:\In\IN^\IN\to\IN^\IN,p\mapsto p-1$ satisfies $\delta_X^\wp(p)=\delta_XG(p)$ and hence it witnesses $\delta_X^\wp\leq\delta_X$.
Altogether $\delta_X^\wp\equiv\delta_X$. We need to prove that $\delta_X^\wp$ is precomplete. Let $F:\In\IN^\IN\to\IN^\IN$ be computable and let
$M$ be a Turing machine that computes $F$. We modify this machine such that it never halts and after every $n$ steps for some suitable fixed number $n\in\IN$
the machine writes a $0$ on the output tape, irrespective of the input. Otherwise the machine is left unchanged. Then the modified machine computes a total
function $G:\IN^\IN\to\IN^\IN$ with $F(p)-1=G(p)-1$ and hence $\delta_X^\wp F(p)=\delta_X^\wp G(p)$ for all $p\in\dom(F)$.
\end{proof}

%There are also many natural examples for precomplete representations, for instance 
%it is not hard to see that the
%standard representation of a second-countable $T_0$--space is precomplete, if defined appropriately (see~\cite[Lemma~3.4.8~(6)]{Wei87}).
%
%\begin{example}
%\label{ex:T0} 
%For every second-countable $T_0$--space $X$ with a countable subbase $(U_n)_{n\in\IN}$
%we can define a representation $\delta_X:\In\IN^\IN\to X$ by
%\[\delta_X(p)=x:\iff\{n\in\IN:x\in U_n\}=\{n\in\IN:n+1\in\range(p)\}.\]
%The representation $\delta_X$ is precomplete.
%\end{example}

We will also need the fact that other classes of functions can be extended to total
ones under precomplete representations. Hence we introduce the following concept.

\begin{definition}[Respect for precompleteness]
We say that a set $P$ of functions $F:\In\IN^\IN\to\IN^\IN$ {\em respects precompleteness},
if for every precomplete representation $\delta$ and any function $F\in P$ there
exists a total function $G\in P$ such that $\delta F(p)=\delta G(p)$ for all $p\in\dom(F)$.
\end{definition}

It is clear that the set of computable functions respects precompleteness by definition.
However, also other classes of functions do. 
Some of them, simply because they can already be extended to total functions in the same class irrespectively of the representation.
We provide a number of examples.
We call a function {\em non-uniformly computable} if it maps all computable inputs in its domain
to computable outputs. By $\J:\IN^\IN\to\IN^\IN,p\mapsto p'$ we denote the
{\em Turing jump operator} and by $U:\In\IN^\IN\to\IN^\IN$ a {\em universal computable function}
such that for every continuous function $F:\In\IN^\IN\to\IN^\IN$ there is a $q\in\IN^\IN$
with $F(p)=U\langle q,p\rangle$ for all $p\in\dom(F)$~\cite[Theorem~2.3.8]{Wei00}.

\begin{proposition}[Respect for precompleteness]
\label{prop:respect-precompleteness}
The following classes of partial functions $F:\In\IN^\IN\to\IN^\IN$ respect precompleteness:
computable, continuous, limit computable, Borel measurable and non-uniformly computable.
\end{proposition}
\begin{proof}
The statement for computable functions is a consequence of the definition of precompleteness.
Let $\delta$ be a precomplete representation and let $F:\In\IN^\IN\to\IN^\IN$ be a continuous function.
Then there is a $q\in\IN^\IN$ such that $F(p)=U\langle q,p\rangle$ for all $p\in\dom(F)$.
By precompleteness of $\delta$, there is a total function $V:\IN^\IN\to\IN^\IN$
with $\delta V(p)=\delta U(p)$ for all $p\in\dom(U)$. 
Hence $G(p):=V\langle q,p\rangle$ defines a total continuous function with
$\delta G(p)=\delta F(p)$ for all $p\in\dom(F)$. 
This shows that the class of continuous functions respects precompleteness.
For every limit computable function $F:\In\IN^\IN\to\IN^\IN$
there exists a computable function $H:\In\IN^\IN\to\IN^\IN$ such that
$F=H\circ\J$~\cite[Theorem~14]{Bra18}.
By precompleteness of $\delta$ there exists a total computable function $I:\IN^\IN\to\IN^\IN$
such that $\delta I(p)=\delta H(p)$ for all $p\in\dom(H)$. 
Hence $G:=I\circ\J$ is a total function that is limit computable
and satisfies $\delta F(p)=\delta G(p)$ for all $p\in\dom(F)$. 
Hence the class of limit computable functions respects precompleteness.
Every partial Borel measurable function $F:\In\IN^\IN\to\IN^\IN$ can be extended
to a total Borel measurable function by a theorem of Kuratowski (see \cite[Theorem~2.2]{Kec95}).
The class of non-uniformly computable functions respects precompleteness
since every non-uniformly computable function $F:\In\IN^\IN\to\IN^\IN$ can
simply be extended to a total non-uniformly computable function
$G:\IN^\IN\to\IN^\IN$ by defining $G(p)=\widehat{0}=000...$ for all $p\in\IN^\IN\setminus\dom(F)$.
\end{proof}

The proof for limit computable functions (which are exactly the effectively $\SO{2}$--computable functions) can easily be extended to any finite level of the Borel hierarchy.
We prove in \cite[Corollaries~8.4, 9.3]{BG19} that functions that are computable with finitely many mind changes
and low functions do not respect precompleteness.

We also need to study how certain algebraic constructions on represented spaces
behave with respect to precompleteness. For any sets $X$ and $Y$ we denote by 
$X\times Y$ and $X^\IN$ the usual {\em products}, by $X\sqcup Y:=(\{0\}\times X)\cup(\{1\}\times Y)$
the {\em disjoint union} of $X$ and $Y$, by $X^*:=\bigcup_{i=0}^\infty(\{i\}\times X^i)$ the {\em set of words} over $X$,
where $X^i$ denotes the $i$--fold product of $X$ with itself, and $X^0:=\{0\}$.
By $\overline{X}:=X\cup\{\bot\}$ we denote the {\em completion} $X$, 
where we assume that $\bot\not\in X$.

\begin{definition}[Constructions on representations]
\label{def:operations-representations}
Let $(X,\delta_X)$ and $(Y,\delta_Y)$ be represented spaces. We define
\begin{enumerate}
\item $\delta_{X\times Y}:\In\IN^\IN\to X\times Y$, $\delta_{X\times Y}\langle p,q\rangle:=(\delta_X(p),\delta_Y(q))$
\item $\delta_{X\sqcup Y}:\In\IN^\IN\to X\sqcup Y$, $\delta_{X\sqcup Y}(0p):=(0,\delta_X(p))$ and $\delta_{X\sqcup Y}(1p):=(1,\delta_Y(p))$
\item $\delta_{X^*}:\In\IN^\IN\to X^*$, $\delta_{X^*}(n\langle p_1,p_2,...,p_n\rangle):=(n,(\delta_X(p_1),\delta_X(p_2),...,\delta_X(p_n)))$
\item $\delta_{X^\IN}:\In\IN^\IN\to X^\IN$, $\delta_{X^\IN}\langle p_0,p_1,p_2,...\rangle:=(\delta_X(p_n))_{n\in\IN}$ 
\item $\delta_{\overline{X}}:\IN^\IN\to\overline{X}$, $\delta_{\overline{X}}(p):=\delta_X^\wp(p)$ if $p\in\dom(\delta_X^\wp)$ and $\delta_{\overline{X}}(p):=\bot$ otherwise.
\end{enumerate}
\end{definition}

We warn the reader that all these constructions on represented spaces preserve equivalence of representations, except the last one
for the completion. In other words, the equivalence class of $\delta_{\overline{X}}$ does not only depend on the equivalence class of
$\delta_X$, but on the concrete representative $\delta_X$ itself. For our applications this does not cause any problems
(see the remark after Corollary~\ref{cor:monotonicity-completion}; the problem could also be circumvented by moving to multi-valued representations
\cite[Lemma~4.2.11]{Sch02c}).

The next observation is that finite and countable products preserve precompleteness.

\begin{proposition}[Products and precompleteness]
\label{prop:products-precompleteness}
Let $(X,\delta_X)$ and $(Y,\delta_Y)$ be precomplete represented spaces.
Then so are $(X\times Y,\delta_{X\times Y})$ and $(X^\IN,\delta_{X^\IN})$.
\end{proposition}
\begin{proof}
If $F:\In\IN^\IN\to\IN^\IN$ is computable, then so are the projections
$F_i=\pi_i\circ F$ for $i\in\{1,2\}$ with $\pi_1\langle p,q\rangle=p$ and $\pi_2\langle p,q\rangle=q$.
Hence, by precompleteness there are total computable functions $G_i:\IN^\IN\to\IN^\IN$
with $\delta_X F_1(p)=\delta_X G_1(p)$ for all $p\in\dom(F_1)$ and with an analogous statement
for $\delta_Y,F_2$ and $G_2$. Let $G(p)=\langle G_1(p),G_2(p)\rangle$ for all $p\in\IN^\IN$.
Then $G:\IN^\IN\to\IN^\IN$ is computable and total, and we obtain $\delta_{X\times Y}F(p)=\delta_{X\times Y}G(p)$ for all $p\in\dom(F)$. 
Hence $\delta_{X\times Y}$ is precomplete.
If $F:\In\IN^\IN\to\IN^\IN$ is computable, then so is the function
$H:\In\IN^\IN\to\IN^\IN,\langle i,p\rangle\mapsto \pi_i\circ F(p)$,
where $\pi_i:\IN^\IN\to\IN^\IN,\langle p_0,p_1,p_2,...\rangle\mapsto p_i$ denotes
the $i$--th projection. Due to precompleteness of $\delta_X$ there is a total computable
function $I:\IN^\IN\to\IN^\IN$ with $\delta_X H(p)=\delta_X I(p)$ for all $p\in\dom(H)$.
Then also the function $G:\IN^\IN\to\IN^\IN,p\mapsto\langle I\langle 0,p\rangle,I\langle 1,p\rangle,I\langle 2,p\rangle,...\rangle$
is computable and total and satisfies $\delta_{X^\IN}F(p)=\delta_{X^\IN}G(p)$ for all $p\in\dom(F)$.
This shows that $\delta_{X^\IN}$ is precomplete.
\end{proof}

The coproduct constructions for $X\sqcup Y$ and $X^*$ are less nicely behaved with respect to precompleteness.
One problem is that also the natural number component that selects the argument 
has to be handled in a precomplete manner. One can modify the definition of
$\delta_{X\sqcup Y}$ and $\delta_{X^*}$ to take this into account.
However, even then it is not clear why the construction should preserve precompleteness.
We just obtain that if $\delta_X$ and $\delta_Y$ are the precompletions 
according to Proposition~\ref{prop:precompleteness}, then $\delta_{X\sqcup Y}$ and
$\delta_{X^*}$ are precomplete in the modified definition. We formulate this more formally.
We use the total representation $\delta_\IN$ of $\IN$ given by $\delta_\IN(p):=p(0)$.

\begin{proposition}[Coproducts and precompleteness]
\label{prop:coproducts-precompleteness}
Let $(X,\delta_X)$ and $(X_i,\delta_{X_i})$ be represented spaces for $i\in\{0,1\}$.
\begin{enumerate}
\item We define a representation $\delta$ of $X_0\sqcup X_1$ by
\[\delta\langle q,p\rangle:=(\delta_\IN^\wp(q),\delta_{X_i}^\wp(p))\]
for all $q,p\in\IN^\IN$ such that $\delta_\IN^\wp(q)=i\in\{0,1\}$ and $p\in\dom(\delta_{X_i}^\wp)$.
Then $\delta$ is precomplete and $\delta\equiv\delta_{X_0\sqcup X_1}$.
\item We define a representation $\delta$ of $X^*$ by
\[\delta\langle q,\langle p_1,...,p_n\rangle\rangle:=(\delta_\IN^\wp(q),(\delta_X^\wp(p_1),...,\delta_X^\wp(p_n)))\]
for all $q,p_1,...,p_n\in\IN^\IN$ such that $\delta_\IN^\wp(q)=n$ and $p_i\in\dom(\delta_X^\wp)$ for 
$i=1,...,n$.
Then $\delta$ is precomplete and $\delta\equiv\delta_{X^*}$.
\end{enumerate}
\end{proposition}
\begin{proof}
The proof is similar to the proof of Proposition~\ref{prop:precompleteness}.
We only consider the case of $X^*$ and leave the case $X_0\sqcup X_1$ to the reader.
Given a $\delta_{X^*}$--name $\langle n,\langle p_1,...,p_n\rangle\rangle$
of $x\in X^*$
we compute $q:=\widehat{n}=nnn...$ and then
$\langle q+1,\langle p_1+1,...,p_n+1\rangle\rangle$ is a $\delta$--name of the same
point $x$. Since $r\mapsto r+1$ is computable, we obtain $\delta_{X^*}\leq\delta$.
Given a $\delta$--name $\langle q,\langle p_1,...,p_n\rangle\rangle$
of a point $x\in X^*$, we can search for the first non-zero value $k\in\IN$ in $q$,
in which case we know that $n=k-1$, and then we can compute $\langle n,\langle p_1-1,...,p_n-1\rangle\rangle$,
which is a $\delta_{X^*}$--name of the same point $x$. Since $r\mapsto r-1$ is computable on
sequences such that $r-1\in\IN^\IN$, we obtain $\delta\leq\delta_{X^*}$.
Any machine that computes a function $F:\In\IN^\IN\to\IN^\IN$ can be modified
as in the proof of Proposition~\ref{prop:precompleteness} such that it computes a
total function $G:\IN^\IN\to\IN^\IN$, potentially with extra zeros on the output side
and such that $\delta F(p)=\delta G(p)$ for all $p\in\dom(F)$.
\end{proof}

We mention that the completion $(\overline{X},\delta_{\overline{X}})$ of a represented
space is always precomplete.
This follows like in the proof of Proposition~\ref{prop:precompleteness}.
The only additional observation required in the proof is that if $\delta_{\overline{X}}F(p)=\bot$, then also $\delta_{\overline{X}}G(p)=\bot$.
We recall that a {\em computable embedding} $f:X\to Y$ is a computable function that is injective
and whose partial inverse $f^{-1}:\In Y\to X$ is computable too.

\begin{corollary}[Completion]
\label{cor:completion}
$(\overline{X},\delta_{\overline{X}})$ is a precomplete represented space for every represented space $(X,\delta_X)$
and $\iota:X\to\overline{X},x\mapsto x$ is a computable embedding.
\end{corollary}

\section{Total Weihrauch Reducibility}
\label{sec:total}

In this section we are going to introduce a total variant of Weihrauch reducibility
that behaves very similarly to the usual reducibility from a practical perspective,
but that has different algebraic properties.

By a {\em problem} $f:\In X\mto Y$ we mean a partial multi-valued map $f:\In X\mto Y$
on represented spaces $(X,\delta_X)$ and $(Y,\delta_Y)$. 
We recall that {\em composition} of problems $f:\In X\mto Y$ and $g:\In Y\mto Z$
is defined by 
\[g\circ f(x):=\{z\in Z:(\exists y\in f(x))\; z\in g(y)\}\]
for all $x\in\dom(g\circ f):=\{x\in\dom(f):f(x)\In\dom(g)\}$.
For two problems $f:\In X\mto Y$ and $g:\In X\mto Z$ with identical source space $X$ we define
the {\em juxtaposition} $(f,g):\In X\mto Y\times Z$ by $(f,g)(x):=f(x)\times g(x)$ and $\dom(f,g):=\dom(f)\cap\dom(g)$.
If $f,g:\In \IN^\IN\mto\IN^\IN$ are problems on Baire space, then we also call $\langle f,g\rangle:=\langle\;\rangle\circ(f,g)$ the
{\em juxtaposition} of $f$ and $g$ and $\langle f\times g\rangle$ defined by $\langle f\times g\rangle\langle p,q\rangle:=\langle f(p),g(q)\rangle$ for all $p,q\in\IN^\IN$ the {\em product} of $f$ and $g$.

We say that a function $F:\In\IN^\IN\to\IN^\IN$ is a {\em realizer} of $f$,
if $\delta_YF(p)\in f\delta_X(p)$ for all $p\in\dom(f\delta_X)$.
We denote this by $F\vdash f$. 
We say that $f$ is {\em computable} if it has a computable realizer.
Other notions, such as continuity, Borel measurability and so forth that are
well-defined for functions $F:\In\IN^\IN\to\IN^\IN$ are transferred in an
analogous manner to problems $f:\In X\mto Y$.

We write $F\vdash_\t f$, if $F$ is a total realizer of $f$.
We now recall the definition of ordinary and strong Weihrauch reducibility on problems $f,g$,
which is denoted by $f\leqW g$ and $f\leqSW g$, respectively, and we introduce
two new concepts of {\em total Weihrauch reducibility} and {\em strong total Weihrauch reducibility},
which are denoted by $f\leqTW g$ and $f\leqSTW g$, respectively.

\begin{definition}[Weihrauch reducibility]
\label{def:Weihrauch-reducibility}
Let $f:\In X\mto Y$ and $g:\In U\mto V$ be problems. We define:
\begin{enumerate}
\item $f\leqW g:\!\iff\!(\exists$ computable $H,K:\In\IN^\IN\to\IN^\IN)(\forall G\vdash g)\; H\langle \id,GK\rangle\vdash f$.
\item $f\leqSW g:\!\iff\!(\exists$ computable $H,K:\In\IN^\IN\to\IN^\IN)(\forall G\vdash g)\;HGK\vdash f$.
\item $f\leqTW g:\!\iff \!\!(\exists$ computable $H,K:\In\IN^\IN\to\IN^\IN)(\forall G\vdash_\t g)\; H\langle \id,GK\rangle\vdash_\t f$.
\item $f\leqSTW g:\!\iff (\exists$ computable $H,K:\In\IN^\IN\to\IN^\IN)(\forall G\vdash_\t g)\;HGK\vdash_\t f$.
\end{enumerate}
For (3) and (4) we assume that we replace each of the given representations of $X,Y,U$ and $V$ by a computably equivalent precomplete representation
of the corresponding set. 
\end{definition}

We call the reducibilities $\leqW$ and $\leqSW$ {\em partial}\footnote{This is not related to the preorders being partial or total in an order theoretic sense; they are both partial in that sense.} 
in order to 
distinguish them from their total counterparts $\leqTW$ and $\leqSTW$.
We note that precompleteness is not required or relevant in the partial case,
but it can be assumed without loss of generality since the concept of partial
(strong) Weihrauch reducibility is invariant under computably equivalent
representations~\cite[Lemma~2.11]{BG11}. In the total cases (3) and (4), 
however, precompleteness is essential,
since otherwise these definitions would not be invariant under computably
equivalent representations. By Proposition~\ref{prop:precompleteness}
we can always choose precomplete representations that are computably equivalent
to the given representations of the spaces $X,Y,U$ and $V$.
But we still need to show that the definition of $\leqTW$ and $\leqSTW$
does not depend on this choice.

We will prove a slightly more general result that highlights the places where precompleteness is actually needed.
For this purpose we introduce the following terminology: we say that 
$f\leqTW g$ {\em holds with respect to} $(\delta_X,\delta_Y,\delta_U,\delta_V)$,
if Definition~\ref{def:Weihrauch-reducibility} (3) holds as it stands but exactly
for the given representations of $X, Y, U$ and $V$, respectively, and these representations are
not required to be precomplete. Hence the statement defined here is weaker than $f\leqTW g$ in the sense defined above.
We use a corresponding terminology for $\leqSTW$.
Now we obtain the following result.

\begin{lemma}[Invariance under representations]
\label{lem:invariance}
Let $f:\In X\mto Y$, $g:\In U\mto V$ be problems on represented
spaces $(X,\delta_X)$, $(Y,\delta_Y)$, $(U,\delta_U)$ and $(V,\delta_V)$.
Let $\delta_X',\delta_Y',\delta_U'$ and $\delta_V'$ be further representations
of the given sets, respectively, such that
\begin{enumerate}
\item $\delta_X'\leq\delta_X$, $\delta_Y\leq\delta_Y'$, $\delta_U\leq\delta_U'$ and     
       $\delta_V'\leq\delta_V$,
\item $\delta_V,\delta_Y'$ and $\delta_U'$ are precomplete.
\end{enumerate}
If $f\leqTW g$ holds with respect to $(\delta_X,\delta_Y,\delta_U,\delta_V)$,
then it also holds with respect to $(\delta_X',\delta_Y',\delta_U',\delta_V')$.
An analogous statement holds for $f\leqSTW g$.
\end{lemma}
\begin{proof}
We follow the construction as outlined in the proof of~\cite[Lemma~2.11]{BG11}.
Since $\delta_U'$ and $\delta_V'$ are precomplete according to (2), we can additionally assume that the computable functions
$S, T:\IN^\IN\to\IN^\IN$ in that proof are total. 
In that proof it is shown that whenever $G'\vdash g$ holds with respect to $(\delta_U',\delta_V')$,
then $G:=TG'S\vdash g$ holds with respect to $(\delta_U,\delta_V)$. 
Due to totality of $T, S$, the same holds true if we replace $\vdash$ by $\vdash_\t$ in both
occurrences. If we assume that $H\langle\id, GK\rangle\vdash_\t f$ holds with respect
to $(\delta_X,\delta_Y)$, then we obtain as in the proof mentioned above that
$H'\langle\id,G'K'\rangle\vdash f$ holds with respect to $(\delta_X',\delta_Y')$.
Due to precompleteness of $\delta_U'$ and $\delta_Y'$ according to (2), we can
always assume that $H',K'$ are even total computable functions. 
Hence, we even obtain $H'\langle\id,G'K'\rangle\vdash_\t f$,
which completes the proof. The proof for $\leqSTW$ is analogous.
\end{proof}

If $f\leqTW g$ holds with respect to $(\delta_X,\delta_Y,\delta_U,\delta_V)$ and at least $\delta_V$ is precomplete
among these representations, then according to Lemma~\ref{lem:invariance} 
we can always replace the non-precomplete representations by equivalent
precomplete ones and $f\leqTW g$ holds with respect to these precomplete representations and
hence $f\leqTW g$ holds in terms of Definition~\ref{def:Weihrauch-reducibility}.

For the moment Lemma~\ref{lem:invariance} is useful as it implies that
$\leqTW$ and $\leqSTW$ are well-defined and invariant under computably equivalent representations.

\begin{corollary}[Invariance under equivalent representations]
\label{cor:invariance}
Let $f:\In X\mto Y$ and $g:\In U\mto V$ be problems.
The relations $f\leqW g$, $f\leqSW g$, $f\leqTW g$ and $f\leqSTW g$ 
remain unchanged if we replace the representations of $X, Y, U$ and $V$
by computably equivalent ones.
\end{corollary}

We note that the statement for $\leqW$ and $\leqSW$ was proved in~\cite[Lemma~2.11]{BG11}.
The following example shows that precompleteness in Definition~\ref{def:Weihrauch-reducibility} 
cannot be omitted if one wants to achieve invariance under equivalent representations.

\begin{example}
Every computable function $F:\In\IN^\IN\to\IN^\IN$ without total computable extension (see Example~\ref{ex:precompleteness})
has a total computable realizer with respect to $(\id,\id^\wp)$, but not with respect to $(\id,\id)$.
Hence $F\leqTW F$ does not hold with respect to $(\id,\id,\id,\id^\wp)$.
Clearly, $F\leqTW F$ holds with respect to $(\id^\wp,\id^\wp,\id^\wp,\id^\wp)$ and hence $F\leqTW F$ holds
in terms of Definition~\ref{def:Weihrauch-reducibility}.
\end{example}

The argument used in the proof of Lemma~\ref{lem:invariance} concerning $H'$ and $K'$
also allows us to slightly rephrase Definition~\ref{def:Weihrauch-reducibility}.
Due to precompleteness we can demand total $H,K$ (and replace $\vdash_\t$ by $\vdash$ on the right-hand side.)

\begin{lemma}[Weihrauch reducibility]
\label{lem:total}
Let $f:\In X\mto Y$ and $g:\In U\mto V$ be problems. 
We choose precomplete representations
that are computably equivalent to the given representations of $X,Y,U$ and $V$.
Then:
\begin{enumerate}
\item $f\leqW g\iff (\exists$ computable $H,K:\IN^\IN\to\IN^\IN)(\forall G\vdash g)\; H\langle \id,GK\rangle\vdash f$.
\item $f\leqSW g\iff (\exists$ computable $H,K:\IN^\IN\to\IN^\IN)(\forall G\vdash g)\;HGK\vdash f$.
\item $f\leqTW g\iff (\exists$ computable $H,K:\IN^\IN\to\IN^\IN)(\forall G\vdash_\t g)\; H\langle \id,GK\rangle\vdash f$.
\item $f\leqSTW g\iff (\exists$ computable $H,K:\IN^\IN\to\IN^\IN)(\forall G\vdash_\t g)\; HGK\vdash f$.
\end{enumerate}
\end{lemma}

The proof of the backward direction is immediate and the forward direction follows
from precompleteness of the representations of $U$ and $Y$, respectively. 

In~\cite[Lemma~2.4]{BG11} we have proved that $\leqW$ and $\leqSW$ are {\em preorders},
i.e., they are reflexive and transitive. The associated equivalences are denoted by $\equivW$ and
$\equivSW$, respectively.
Using Lemma~\ref{lem:total} we can now easily
transfer these proofs to the case of the total reducibilities.

\begin{proposition}[Preorders]
\label{prop:preorders}
The relations $\leqTW$ and $\leqSTW$ are preorders on the class of problems.
\end{proposition}
\begin{proof}
We follow the proof of~\cite[Lemma~2.4]{BG11} and the notations used therein.
Reflexivity is obvious as the corresponding functions $H,K$ are total.
For the transitivity proof, we assume that the reductions $f\leqTW g$ and $g\leqTW h$
are given by total $H,K,H',K'$. Then the corresponding functions $H''$ and $K''$
constructed in the proof of~\cite[Lemma~2.4]{BG11} are also total and
hence the claim follows from Lemma~\ref{lem:total}.
\end{proof}

By $\equivTW$ and $\equivSTW$ we denote the equivalence relations
that are associated with $\leqTW$ and $\leqSTW$, respectively.
If the different versions of Weihrauch reducibility are expressed as in Lemma~\ref{lem:total}, 
then it is immediately clear that a partial
reduction implies the corresponding total reduction.
Using Lemma~\ref{lem:total}, Corollary~\ref{cor:invariance} and Proposition~\ref{prop:precompleteness} obtain the following corollary.

\begin{corollary}[Partial and total Weihrauch reducibility]
\label{cor:partial-total}
Let $f$ and $g$ be problems. Then $f\leqW g\TO f\leqTW g$ and $f\leqSW g\TO f\leqSTW g$.
\end{corollary}

This means that all positive results that hold for a partial version of Weihrauch
reducibility can be transferred to the corresponding total variant.
Together with the obvious other implications we obtain the diagram 
for the logical relations between different versions of Weihrauch reducibility that is displayed in Figure~\ref{fig:reducibility}. The diagram is complete up to transitivity (see Example~\ref{ex:reducibility}).
The diagram also shows the generating closure operators of cylindrification and completion that we discuss later.

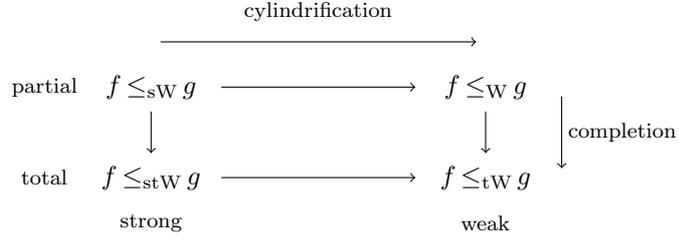
\begin{figure}[htb]
\begin{center}
\begin{tikzpicture}[scale=.4,auto=left]
\useasboundingbox  rectangle (27.5,8);

\node at (9.5,5) {$f\leqSW g$};
\node at (20.5,5) {$f\leqW g$};
\node at (20.5,2) {$f\leqTW g$};
\node at (9.5,2) {$f\leqSTW g$};
%\node at (27,2) {$f\leq_{\rm\omega} g$};
%\node at (20,4) {$f\leq_{\rm gW} g$};

\node (v1) at (9.5,4.5) {};
\node (v2) at (18.5,5) {};
\node (v3) at (20.5,4.5) {};
%\node (v4) at (18,4) {};
\node (v5) at (11.5,2) {};
\node (v6) at (20.5,2.5) {};
\node (v7) at (9.5,2.5) {};
%\node (v8) at (22,4) {};
%\node (v10) at (22,2) {};
%\node (v9) at (25,2) {};
%\draw [->] (v1) edge (v2);
%\draw [->] (v3) edge (v4);
%\draw [->] (v5) edge (v6);
\draw [->] (v1) edge (v7);
%\draw [->] (v8) edge (v9);
%\draw [->] (v10) edge (v9);
\draw [->] (v3) edge (v6);

\node (v12) at (6,5) {\footnotesize partial};
\node at (9.5,0.5) {\footnotesize strong};
\node (v11) at (6,2) {\footnotesize total};
\node at (20.5,0.5) {\footnotesize weak};
\node (v4) at (23,5) {};
\node (v8) at (23,2) {};
\node (v9) at (9.5,6.5) {};
\node (v10) at (20.5,6.5) {};
\draw [double,->] (v4) edge (v8);
\draw [->] (v9) edge (v10);
\node at (15,7.5) {\footnotesize cylindrification};
\node at (25,3.5) {\footnotesize completion};
\node (v13) at (18.5,2) {};
\draw [->] (v5) edge (v13);
\node (v14) at (11.5,5) {};
\draw [->] (v14) edge (v2);
\end{tikzpicture}
\end{center}
\ \\[-0.3cm]
\caption{Implications between notions of reducibility}
\label{fig:reducibility}
\end{figure}

\begin{example}
\label{ex:reducibility}
Let $f:\IN^\IN\to\IN^\IN$ denote a constant function with computable value.
Then $\id\equivW f$, but $\id\nleqSTW f$. Let $0:\In\IN^\IN\to\IN^\IN$ denote
the nowhere defined function. Then $\id\equivTW 0$, but $\id\nleqW 0$.
Let $\id|_{\{p\}}:\In\IN^\IN\to\IN^\IN$ be the identity restricted to a non-computable $p\in\IN^\IN$.
Then $\id\equivSTW \id|_{\{p\}}\times\id$, but $\id\nleqW \id|_{\{p\}}\times\id$.
\end{example}

We note that the reducibilities $\leqTW$ and $\leqSTW$ share similar
properties as $\leqW$ and $\leqSW$ when it comes to the preservation of 
computability or other properties.
We say that a class $\CC$ of problems is {\em preserved downwards} by a reducibility $\leq$ for problems
if $f\leq g$ and $g\in\CC$ imply $f\in\CC$.

\begin{proposition}[Downwards preservation]
\label{prop:downwards-preservation}
Computability, continuity, limit computability, Borel measurability and non-uniform computability are preserved
downwards by $\leqTW$.
\end{proposition}
\begin{proof}
Let $\CC$ be the class of computable, continuous, limit computable, Borel measurable or non-uniformly computable problems.
We choose precomplete representations and total computable $H,K$
that witness $f\leqTW g$ according to Lemma~\ref{lem:total}.
If $g\in\CC$, then it has a realizer $G:\In\IN^\IN\to\IN^\IN$ that is in $\CC$.
Since the target space of $g$ is represented with a precomplete representation,
we can assume without loss of generality that $G$ is total by Proposition~\ref{prop:respect-precompleteness}. 
Hence $H\langle\id,GK\rangle$ is a (even total) realizer of $f$ that is also in the class $\CC$.
This proves that $f\in\CC$.
\end{proof}

Any class $\CC$ of functions $F:\In\IN^\IN\to\IN^\IN$ constitutes a property of problems that is preserved 
downwards by total Weihrauch reducibility if the following conditions are satisfied:
$\CC$ contains the identity, is closed under composition with computable functions, is closed under juxtaposition with
the identity and $\CC$ respects precompleteness. In \cite[Corollaries~6.2, 7.4, 8.3]{BG19} we prove that finite mind change computability
and Las Vegas computability is not preserved downwards by $\leqTW$, whereas non-deterministic computability is preserved.

It is known that the class of the nowhere defined problems (often denoted by ${\mathbf 0}$)
forms the bottom element of the Weihrauch lattice~\cite[Lemma~2.7]{BG11}, while the Weihrauch
equivalence class of $\id$ (often denoted by ${\mathbf 1}$) is the class of all computable 
problems with at least one computable point in the domain~\cite[Lemma~2.8]{BG11}.
Moreover, a problem $f$ is computable if and only if $f\leqW\id$.
The statement about the nowhere defined function $0:\In\IN^\IN\to\IN^\IN$ in Example~\ref{ex:reducibility}, namely that $\id\equivTW 0$,
together with Proposition~\ref{prop:downwards-preservation}
show that 
the minimal equivalence class with respect to total Weihrauch reducibility is the class of all computable problems.

\begin{corollary}[Minimal total degree]
\label{cor:minimal}
The equivalence class of all computable problems forms the minimal element
with respect to total Weihrauch reducibility.
\end{corollary}

This already shows that the algebraic structure induced by total Weihrauch reducibility
is significantly different from the structure induced by partial Weihrauch reducibility.
In between ${\mathbf 0}$ and ${\mathbf 1}$ one obtains a complicated structure
for partial Weihrauch reducibility, and among other results one can show that
one can embed the entire Medvedev lattice (and hence the Turing semi-lattice)
in an order-reversing way into the Weihrauch lattice between ${\mathbf 0}$ and ${\mathbf 1}$~\cite[Lemma~5.6]{HP13}.
In contrast to this the two degrees ${\mathbf 0}$ and ${\mathbf 1}$ fall together
with respect to total Weihrauch reducibility. 

Strictly speaking, the class of problems is not a set, but we can always consider
representatives of problems on Baire space to obtain a set as underlying structure.
This is known for $\leqW$ and $\leqSW$ (see~\cite[Lemma~3.8]{BGP18})
and holds correspondingly for $\leqTW$ and $\leqSTW$.

\begin{corollary}[Realizer version]
\label{cor:realizer-version}
Let $f:\In X\mto Y$ be a problem on represented spaces $(X,\delta_X)$ and $(Y,\delta_Y)$.
Then $f^\r:=\delta_Y^{-1}\circ f\circ\delta_X:\In\IN^\IN\mto\IN^\IN$ satisfies $f^\r\equivSTW f$.
\end{corollary}
\begin{proof}
By $f^\r\equivSW f$ holds according to \cite[Lemma~3.8]{BGP18} (and is easy to see, since
$f^\r$ and $f$ share exactly the same realizers). 
Hence $f^\r\equivSTW f$ follows by Corollary~\ref{cor:partial-total}.
\end{proof}

We note that we do not need to assume that $\delta_X$ and $\delta_Y$ are
precomplete in this result. However, for $f^\r:\In\IN^\IN\mto\IN^\IN$
we need to use precomplete representations of $\IN^\IN$ for the
total versions of Weihrauch reducibility.

\section{Completion}
\label{sec:completion}

In this section we discuss the closure operation of completion $f\mapsto\overline{f}$ that generates $\leqTW$ on $\leqW$ and $\leqSTW$ on $\leqSW$.
For the definition of the completion $\overline{f}$ we use the completion $\overline{X}$ of a represented space according
to Definition~\ref{def:operations-representations}.

\begin{definition}[Completion]
\label{def:completion}
Let $f:\In X\mto Y$ be a problem. We define the {\em completion} of $f$ by
\[\overline{f}:\overline{X}\mto\overline{Y},x\mapsto\left\{\begin{array}{ll}
   f(x) & \mbox{if $x\in\dom(f)$}\\
   \overline{Y} & \mbox{otherwise}
\end{array}\right.\]
\end{definition}

We note that the completion $\overline{f}$ is always {\em pointed}, i.e., it has a computable
point in its domain. This is because $\bot\in\overline{X}$ is always computable (as it has the constant zero sequence as a name).

Sometimes it is useful to think of $\overline{f}$ in terms of
its realizer version ${\overline{f}\,}^\r:\IN^\IN\mto\IN^\IN$, which is given by
\[{\overline{f}\,}^\r(p)=\delta_{\overline{Y}}^{-1}\circ\overline{f}\circ\delta_{\overline{X}}(p)
=\left\{\begin{array}{ll}
  (\delta_Y^\wp)^{-1}\circ f\circ\delta_X^\wp(p) & \mbox{if $p\in\dom(f\circ\delta_X^\wp)$}\\
  \IN^\IN & \mbox{otherwise}
\end{array}\right..\]
Since $\overline{f}$ has exactly the same realizers as ${\overline{f}\,}^\r$, one
can deduce from this formula that the realizers of $\overline{f}$ are exactly the total realizers of $f$
with respect to $\delta_X^\wp$ and $\delta_Y^\wp$, which immediately yields
the following conclusion with the help of Lemma~\ref{lem:total}.

\begin{lemma}[Completion and total Weihrauch reducibility]
\label{lem:completion-total}
For all problems $f,g$:
$f\leqW\overline{g}\iff\overline{f}\leqW\overline{g}\iff f\leqTW g$ 
and $f\leqSW\overline{g}\iff\overline{f}\leqSW\overline{g}\iff f\leqSTW g$.
\end{lemma}

Thus, we could define total Weihrauch reducibility also using the completion operation and partial Weihrauch reducibility.
Lemma~\ref{lem:completion-total} also shows that the total Weihrauch degrees can be order theoretically embedded into the pointed partial Weihrauch degrees.
Together with Corollary~\ref{cor:partial-total}
we obtain that completion is monotone.

\begin{corollary}[Monotonicity of completion]
\label{cor:monotonicity-completion}
Let $f$ and $g$ be problems. Then
\begin{enumerate}
\item $f\leqW g\TO \overline{f}\leqW\overline{g}$,
\item $f\leqSW g\TO \overline{f}\leqSW\overline{g}$.
\end{enumerate}
\end{corollary}

We note that this result also implies that completion
is a well-defined operation on (strong) Weihrauch degrees: if $f_1,f_2$ are identical problems with possibly
different but computably equivalent representations on the input and output side, respectively, then
$f_1\equivSW f_2$ and hence $\overline{f_1}\equivSW\overline{f_2}$ follows. This is so, even so
the representations of the corresponding completions of the spaces on the input and output side
are not necessarily computably equivalent
(see the remark after Definition~\ref{def:operations-representations}). Now we can see that completion is a closure operator. 

\begin{proposition}[Completion as closure operator]
\label{prop:completion-closure}
Completion $f\mapsto\overline{f}$ is a closure operator on $\leqW$ and $\leqSW$.
\end{proposition}
\begin{proof}
By Lemma~\ref{lem:completion-total} $f\leqSW\overline{f}$ is equivalent to $f\leqSTW f$, which holds since $\leqSTW$ is reflexive by Proposition~\ref{prop:preorders}.
By Lemma~\ref{lem:completion-total} $\overline{\overline{f}}\leqSW\overline{f}$ is equivalent to $\overline{f}\leqSW\overline{f}$, which holds since $\leqSW$ is reflexive.
Completion is monotone with respect to $\leqW$ and $\leqSW$ according to Corollary~\ref{cor:monotonicity-completion}.
Altogether completion is a closure operator with respect to $\leqW$ and $\leqSW$.
\end{proof}

We have used properties of $\leqTW$ and $\leqSTW$ in order to obtain properties of completion.
Vice versa Proposition~\ref{prop:completion-closure} and Lemma~\ref{lem:completion-total} also imply Proposition~\ref{prop:preorders} and Corollary~\ref{cor:partial-total}. Hence, these concepts yield different perspectives on the same properties.

It is clear that every $f$ is strongly totally equivalent to its completion by Lemma~\ref{lem:completion-total} and Proposition~\ref{prop:completion-closure}.

\begin{corollary} 
\label{cor:completion-total}
$f\equivSTW\overline{f}$ for every problem $f$.
\end{corollary}

In the study of total Weihrauch reducibility the degrees that have identical cones
with respect to partial and total Weihrauch reducibility play an important role.
Hence, we introduce a name for such degrees.

\begin{definition}[Complete problems]
A problem $f$ is called {\em complete} if $f\equivW\overline{f}$ and
{\em strongly complete} if $f\equivSW\overline{f}$.
\end{definition}

Now we obtain the following straightforward characterization of completeness.

\begin{theorem}[Completeness]
\label{thm:completeness}
Let $g$ be a problem. Then the following hold:
\begin{enumerate}
\item $g$ complete $\iff (\forall$ problems $f)(f\leqW g\iff f\leqTW g)$.
\item $g$ strongly complete $\iff (\forall$ problems $f)(f\leqSW g\iff f\leqSTW g)$.
\end{enumerate}
\end{theorem}
\begin{proof}
If $g$ is (strongly) complete, then the respective given equivalence 
holds by Lemma~\ref{lem:completion-total}.
On the other hand, if $f\leqW g\iff f\leqTW g$ holds for all $f$,
then $\overline{g}\equivW g$ follows from Corollary~\ref{cor:completion-total}.
The case of strong completeness can be handled analogously.
\end{proof}

Examples of complete problems are abundant. We study a number of landmarks in the Weihrauch lattice,
among them the Turing jump operator $\J$ and and the binary sorting problem $\SORT$ that
was introduced and studied by Neumann and Pauly~\cite{NP18}.
The problems $\WBWT_2,\ACC_X,\PA$ and $\MLR$ were studied for instance in \cite{BHK17a}.
We identify $X\in\IN$ with the set $X=\{0,1,...,X-1\}$.
Many further completeness questions regarding choice are studied in \cite{BG19}.

\begin{proposition}[Complete problems]
\label{prop:complete-problems}
The following problems are all strongly complete:
\begin{enumerate}
\item $\id:\IN^\IN\to\IN^\IN,p\mapsto p$,
\item $\J:\IN^\IN\to\IN^\IN,p\mapsto p'$,
\item $\lim:\In\IN^\IN\to\IN^\IN,\langle p_0,p_1,p_2,...\rangle\mapsto\lim_{n\to\infty} p_n$,
\item $\LPO:\IN^\IN\to\{0,1\},\LPO(p)=0:\iff(\exists n\in\IN)\;p(n)=0$,
\item $\SORT:2^\IN\to2^\IN$ with
   \[\SORT(p):=\left\{\begin{array}{ll}
    0^k\widehat{1} & \mbox{if $p$ contains exactly $k\in\IN$ zeros}\\
    \widehat{0}       & \mbox{if $p$ contains infinitely many zeros}
\end{array}\right..\]
\item $\WBWT_2:2^\IN\mto2^\IN,p\mapsto\{q\in2^\IN:\lim_{n\to\infty}q(n)$ is a cluster point of $p\}$.
\item $\ACC_{X}:\In\IN^\IN\mto\IN,p\mapsto\{n\in X:n+1\not\in\range(p)\}$, where $X\geq2$ or $X=\IN$
         and $\dom(\ACC_X):=\{p\in\IN^\IN:\range(p)\In\{0,n+1\}$ for some $n\in X\}$.
\item $\PA:2^\IN\mto2^\IN,p\mapsto\{q\in 2^\IN:q$ is a PA-degree relative to $p\}$.
\item $\MLR:2^\IN\mto2^\IN,p\mapsto\{q\in 2^\IN:q$ Martin-L\"of random relative to $p\}$.
\end{enumerate}
\end{proposition}
\begin{proof}
(1) Follows since $\overline{\id}\leqSW\id_{\overline{\IN^\IN}}\leqSW\id$.\\
(2) There is a total computable function $H:\IN^\IN\to\IN^\IN$
such that $H\circ\J(p)=\J(p-1)+1$ for all $p$ with $p-1\in\IN^\IN$.
This can be proved using the smn-Theorem.
Together with the identity $K$ this function $H$ witnesses the reduction
${\overline{\J}\,}\leqSW\J$.\\
(3) Follows by Corollary~\ref{cor:monotonicity-completion} since $\lim\equivSW\J$ holds (see \cite[Theorem~6.7]{BGP18}).\\
(4) Given a name $p\in\IN^\IN$ of a point in $\overline{\IN^\IN}$ with respect to $\delta_{\overline{\IN^\IN}}$,
we can compute $K(p)$ as follows: $K(p)(n)=0:\iff p(n)=1$ and $K(p)(n):=1$ otherwise. 
If $\delta_{\overline{\IN^\IN}}(p)=q\in\IN^\IN$, then $\LPO\circ K(p) =\LPO(q)$. Hence, together with $H(r):=r+1$ the functions
$H,K$ witness $\overline{\LPO}\leqSW\LPO$.\\
(5) As always we assume that $2^\IN$ is represented by $\delta_{2^\IN}:\In\IN^\IN\to2^\IN,p\mapsto p$ with $\dom(\delta_{2^\IN})=2^\IN$.
Given a name $p\in\IN^\IN$ of some $q\in\overline{2^\IN}$, i.e., $\delta_{\overline{2^\IN}}(p)=q$ we can compute $K(p)$ as follows:
$K(p)(n)=0:\iff p(n)=1$ and $K(p)(n):=1$ otherwise. Then $\SORT\circ K(p)=\SORT(q)$ if $q\in2^\IN$.
Hence, $H(r):=r+1$ and $K$ witness $\overline{\SORT}\leqSW\SORT$. \\
(6) We represent $2^\IN$ as above. Given a name $p\in\IN^\IN$ of some $q\in\overline{2^\IN}$ we can compute $K(p)$
as follows, we let $K(p)(n):=p(n)-1$ if $p(n)\not=0$ and we let $K(p)(n)=i$ for the number $i\in\{0,1\}$ such that $i+1$ appears
a maximal number of times within $p(0),...,p(n)$ (and we choose $i=0$ if $1$ and $2$ appear equally often).
This construction guarantees that we do not generate any additional cluster points, i.e., $\WBWT_2K(p)=\WBWT_2(q)$ for $q\in2^\IN$.
Similarly as in the other cases above, this proves $\overline{\WBWT_2}\leqSW\WBWT_2$.\\
(7) Given some name $p\in\IN^\IN$ of a point $q\in\overline{\IN^\IN}$ we compute $K(p)$ as follows:
we let $K(p)(n):=k+1$ if $k+2=p(n)$ is the first number larger than $1$ among $p(0),...,p(n)$ and $k\in X$.
Otherwise, we let $K(p)(n):=0$. This guarantees that $\ACC_XK(p)=\ACC_X(q)$, if $q\in\dom(\ACC_X)$. 
Similarly as in the other cases above, this proves $\overline{\ACC_X}\leqSW\ACC_X$.\\
(8), (9) 
We use $K:\IN^\IN\to2^\IN,p\mapsto0^{p(0)+1}10^{p(1)+1}10^{p(2)+1}...$, which is total computable.
It is straightforward to see that every problem $F:2^\IN\mto2^\IN$ that is antitone in the sense
that $p\leqT q$ implies $F(q)\In F(p)$ is strongly complete. This is because $p-1\leqT p\equivT K(p)$ if $p\in\IN^\IN$ is such that $p-1\in2^\IN$, and hence $FK(p)\In F(p-1)$.
This proves $\overline{F}\leqSW F$.
This applies in particular to $\PA$ and $\MLR$.
\end{proof}

These results show that the cones below the given problems are identical in the total and partial Weihrauch lattices.
It is known, for instance, that $f$ is limit computable if and only if $f\leqW\lim$ \cite{BGP18}. 
Hence, an analogous statement holds for $\leqTW$.

\section{Algebraic Operations}
\label{sec:algebra}

In this section we want to discuss properties of certain algebraic operations and
we want to prove that the total versions of Weihrauch reducibility yield
lattice structures. We start recalling the usual algebraic operations on 
the Weihrauch lattice~\cite{BGP18}.

\begin{definition}[Algebraic operations]
\label{def:algebraic-operations}
Let $f:\In X\mto Y$ and $g:\In U\mto V$  
be multi-valued functions. We define
the following operations:
\begin{enumerate}
\item $f\times g:\In X\times U\mto Y\times V, (f\times g)(x,u):=f(x)\times g(u)$ and\\
$\dom(f\times g):=\dom(f)\times\dom(g)$ \hfill (product)
\item $f\sqcup g:\In X\sqcup U\mto Y\sqcup V$, $(f\sqcup g)(0,x):=\{0\}\times f(x)$, $(f\sqcup g)(1,u):=\{1\}\times g(u)$ and
$\dom(f\sqcup g):=\dom(f)\sqcup\dom(g)$\hfill (coproduct)
\item $f\boxplus g:\In X\sqcup U\mto\overline{Y}\times\overline{V}$, $(f\boxplus g)(0,x):=f(x)\times\overline{V}$, $(f\boxplus g)(1,u):=\overline{Y}\times g(u)$ and 
$\dom(f\boxplus g):=\dom(f)\sqcup\dom(g)$ \hfill (box sum)
\item $f\sqcap g:\In X\times U\mto Y\sqcup V, (f\sqcap g)(x,u):=f(x)\sqcup g(u)$ and\\
$\dom(f\sqcap g):=\dom(f)\times\dom(g)$ \hfill (meet)
\item $f+g:\In X\times U\mto \overline{Y}\times \overline{V}, (f+g)(x,u):=(f(x)\times\overline{V})\cup(\overline{Y}\times g(u))$ and\\
$\dom(f+g):=\dom(f)\times\dom(g)$ \hfill (sum)
\item $f^*:\In X^*\mto Y^*,f^*(i,x):=\{i\}\times f^i(x)$ and\\
$\dom(f^*):=\dom(f)^*$ \hfill (finite parallelization)
\item $\widehat{f}:\In X^\IN\mto Y^\IN,\widehat{f}(x_n)_{n\in\IN}:=\bigtimes_{i\in\IN} f(x_i)$ and\\
$\dom(\widehat{f}):=\dom(f)^\IN$ \hfill (parallelization)
\end{enumerate}
\end{definition}

For every operation $\Box\in\{\times,\sqcup,\boxplus,\sqcap,+\}$ we define its {\em completion} $\overline{\Box}$ by $f\overline{\Box} g:=\overline{f}\Box\overline{g}$.
It follows from Lemma~\ref{lem:completion-total} that these operations are monotone with respect to total Weihrauch reducibility, since the underlying operations $\Box$ 
are monotone with respect to partial Weihrauch reducibility by \cite[Proposition~3.6]{BGP18}.

\begin{corollary}[Monotonicity]
\label{cor:monotonicity}
$(f,g)\mapsto f\overline{\Box}g$ for $\Box\in\{\times,\sqcup,\boxplus,\sqcap,+\}$, $f\mapsto\widehat{\overline{f}}$ and $f\mapsto{\overline{f}\,}^*$
 are monotone with respect to $\leqTW$ and $\leqSTW$.
\end{corollary}
\begin{proof}
By Lemma~\ref{lem:completion-total} completion generates $\leqTW$ on $\leqW$ (and $\leqSTW$ on $\leqSW$). 
By \cite[Proposition~3.6]{BGP18} the given operations $\Box$ are monotone with respect to $\leqW$ and $\leqSW$, respectively. 
The claim now follows with Proposition~\ref{prop:closure-operators}.
\end{proof}

Now we prove that the algebraic operations preserve completeness in the sense of Definition~\ref{def:preservation}.
It is clear by Proposition~\ref{prop:preservation-suprema-infima} that we also get co-preservation for suprema (see Proposition~\ref{prop:infima-suprema}).
Later we will show that this also holds for $+$ (see Proposition~\ref{prop:sums}).

\begin{proposition}[Completion and algebraic operations]
\label{prop:completion-algebraic}
Let $f$ and $g$ be problems. We obtain 
\begin{enumerate}
\item $\overline{f\Box g}\leqSW\overline{f}\Box\overline{g}\equivSW\overline{\overline{f}\Box\overline{g}}$ for $\Box\in\{\times,\sqcup,\boxplus,\sqcap,+\}$,
\item $\overline{\widehat{f}}\leqSW\widehat{\overline{f}}\equivSW\overline{\widehat{\overline{f}}}$,
\item $\overline{f^*}\leqSW{\overline{f}\,}^*\equivSW\overline{{\overline{f}\,}^*}$.
\end{enumerate}
In particular, if $f$ and $g$ are (strongly) complete, then so are $f\times g$, $f\sqcup g$, $f\boxplus g$, $f\sqcap g$, $f+g$, $\widehat{f}$ and $f^*$.
\end{proposition}
\begin{proof}
We consider problems $f:\In X\mto Y$ and $g:\In U\mto V$ and $\Box\in\{\times,\sqcup,\boxplus,\sqcap,+\}$.
Since $X\sqcup U\In\overline{X}\sqcup\overline{U}$ and $X\times U\In\overline{X}\times\overline{U}$,
it follows that $\dom(f\Box g)\In\dom(\overline{f}\Box\overline{g})$, and restricted to $x\in\dom(f\Box g)$
we have $\overline{f\Box g}(x)=(f\Box g)(x)\In(\overline{f}\Box\overline{g})(x)$.
The ``$\In$'' is even an equality in the cases $\Box\in\{\times,\sqcup,\sqcap\}$. In the other cases it is not an equality
simply because $\overline{V}\subsetneqq\overline{\overline{V}}$ and $\overline{Y}\subsetneqq\overline{\overline{Y}}$.
We can also assume that the representations of $\overline{X}\times\overline{U}$ and $\overline{X}\sqcup\overline{U}$
are total (since the representations of $\overline{X}$ and $\overline{U}$ are so). Hence every realizer of $\overline{f}\Box\overline{g}$ is total.
By Corollary~\ref{cor:completion} $\iota:Z\to\overline{Z},z\mapsto z$ is computable for every represented space $Z$, hence
it follows that  $\overline{f\Box g}\leqSW\overline{f}\Box\overline{g}$, since a realizer for $\overline{f\Box g}$ can choose any value outside of $\dom(f\Box g)$.
This also holds in the cases where we only have ``$\In$'' above, since the representation of $\overline{V}$ is total, every name of a point in $\overline{\overline{V}}$ 
is also a name of some point in $\overline{V}$
and an analogous statement holds for $Y$. 
The proofs for the unary operations are analogous. We have $X^\IN\In\overline{X}^\IN$ and $X^*\In\overline{X}^*$ and hence
$\overline{\widehat{f}}\leqSW\widehat{\overline{f}}$ and $\overline{f^*}\leqSW{\overline{f}\,}^*$.
The remaining claims follow by Proposition~\ref{prop:preservation} as completion is a closure operator by Proposition~\ref{prop:completion-closure}.
\end{proof}

The closure properties of complete problems are very useful.
For instance, it is known that $\lim\equivSW\widehat{\LPO}$ \cite{BGP18} and hence the statement on $\lim$ in Proposition~\ref{prop:complete-problems}
could also be derived from the statement on $\LPO$.
Likewise, we obtain a number of further complete problems in this way.
We refrain from giving exact definitions of the listed problems, but we rather
point the reader to \cite{BHK17a} were all stated equivalences have been proved \cite[Theorem~5.2, Corollary~5.3, Proposition~14.10]{BHK17a}.
For the purpose of this article, the equivalences can be read as definitions.

\begin{corollary}[Complete problems]
\label{cor:complete-problems}
$\WKL\equivSW\C_{2^\IN}\equivSW\widehat{\ACC_2}$, $\DNC_X\equivSW\widehat{\ACC_X}$ for $X\in\IN$ with $X\geq2$ or $X=\IN$ are strongly complete,
and $\COH\equivW\widehat{\WBWT_2}$ is complete.
\end{corollary}

In~\cite[Proposition~3.11]{BG11} we proved that $\sqcap$ is the infimum operation with respect to $\leqSW$ and $\leqW$.
That $\sqcup$ is the supremum operation with respect to $\leqW$ was first proved by Pauly~\cite[Theorem~4.5]{Pau10a} (see also \cite[Theorem~3.9]{BGP18}).
Dzhafarov proved that $\boxplus$ is a supremum operation for $\leqSW$~\cite{Dzh19} and he also showed $f\boxplus g\equivW f\sqcup g$.
Using Propositions~\ref{prop:closure-operators} and \ref{prop:preservation-suprema-infima} we can transfer these
results to the total versions of Weihrauch reducibility.

\begin{proposition}[Infima and suprema]
\label{prop:infima-suprema}
Let $f,g$ be problems. Then
\begin{enumerate}
\item $\overline{f}\sqcap\overline{g}$ is an infimum of $f$ and $g$ with respect to $\leqTW$ and $\leqSTW$.
\item $f\sqcup g$ is a supremum of $f$ and $g$ with respect to $\leqTW$.
\item $f\boxplus g$ is a supremum of $f$ and $g$ with respect to $\leqSTW$.
\item $\overline{f\sqcup g}\equivW\overline{f}\sqcup\overline{g}$ and hence $f\sqcup g\equivTW\overline{f}\sqcup\overline{g}\equivTW f\boxplus g$.
\item $\overline{f\boxplus g}\equivSW\overline{f}\boxplus\overline{g}$ and hence $f\boxplus g\equivSTW\overline{f}\boxplus\overline{g}$.
\end{enumerate}
\end{proposition}

In Lemma~\ref{lem:counter-monotonicity} we will see that the equivalences in (4) cannot be strengthened
to strong equivalences.

By a {\em (strong) total Weihrauch degree} we mean an equivalence class with respect to $\leqTW$
(or with respect to $\leqSTW$ in the strong case). We denote the corresponding classes
by $\WW_{\mathrm{tW}}$ and $\WW_{\mathrm{stW}}$. 
Strictly speaking, these are not sets, but every equivalence class has a representative on Baire space
according to Corollary~\ref{cor:realizer-version}, and if desired, we can turn the classes
$\WW_{\mathrm{tW}}$ and $\WW_{\mathrm{stW}}$ into sets of such representatives.
The same applies to further classes of degrees that we consider in the following.
We can extend the reducibilities $\leqTW$ and $\leqSTW$ to the corresponding
degrees and any monotone algebraic operation too.
By Proposition~\ref{prop:infima-suprema}
$(\WW_{\mathrm{tW}},\leqTW,\overline{\sqcap},\sqcup)$ yields a lattice structure.

It was first proved by Pauly~\cite[Theorem~4.22]{Pau10a} that the Weihrauch lattice is distributive. 
In fact, he proved that it is a distributive join semi-lattice, which implies distributivity as a lattice.
That is, we have $f\sqcup(g\sqcap h)\equivW(f\sqcup g)\sqcap(f\sqcup h)$ and
$f\sqcap(g\sqcup h)\equivW(f\sqcap g)\sqcup(f\sqcap h)$~\cite[Theorem~31]{BP18}.
Also the total Weihrauch degrees form a distributive lattice.

\begin{theorem}[Total Weihrauch lattice]
\label{thm:total-lattice}
$(\WW_{\mathrm{tW}},\leqTW,\overline{\sqcap},\sqcup)$ is a distributive lattice.
\end{theorem}
\begin{proof}
By Proposition~\ref{prop:infima-suprema} we obtain
\[
f\overline{\sqcap}(g\sqcup h)=\overline{f}\sqcap\overline{(g\sqcup h)}\equivW\overline{f}\sqcap(\overline{g}\sqcup\overline{h})
\equivW(\overline{f}\sqcap\overline{g})\sqcup(\overline{f}\sqcap\overline{h})=(f\overline{\sqcap}g)\sqcup(f\overline{\sqcap}h)
\]
and hence $f\overline{\sqcap}(g\sqcup h)\equivTW(f\overline{\sqcap}g)\sqcup(f\overline{\sqcap}h)$ by Corollary~\ref{cor:partial-total}.
With Proposition~\ref{prop:infima-suprema} and Corollary~\ref{cor:completion-total} we obtain similarly as above
\begin{eqnarray*}
f\sqcup(g\overline{\sqcap} h)
&\equivTW&\overline{f}\sqcup{(g\overline{\sqcap}h)}=\overline{f}\sqcup{(\overline{g}\sqcap\overline{h})}\equivW(\overline{f}\sqcup\overline{g})\sqcap(\overline{f}\sqcup\overline{h})\\
&\equivW&\overline{(f\sqcup g)}\sqcap\overline{(f\sqcup h)}=(f\sqcup g)\overline{\sqcap}(f\sqcup h)
\end{eqnarray*}
and hence $f\sqcup(g\overline{\sqcap} h)\equivTW(f\sqcup g)\overline{\sqcap}(f\sqcup h)$. 
Altogether, this shows that the total Weihrauch lattice is distributive.
\end{proof}

Proposition~\ref{prop:infima-suprema} implies that
$\WW_{\mathrm{stW}}$ is a lattice. 
Dzhafarov proved that the lattice $\WW_{\mathrm{sW}}$
is not distributive~\cite[Theorem~4.4]{Dzh19}. 
We can transfer his proof to $\WW_{\mathrm{stW}}$.

\begin{theorem}[Strong total Weihrauch lattice]
\label{thm:strong-total-lattice}
$(\WW_{\mathrm{stW}},\leqSTW,\overline{\sqcap},\boxplus)$ is a lattice, which is not distributive.
\end{theorem}
\begin{proof}
Proposition~\ref{prop:infima-suprema} implies that
$\WW_{\mathrm{stW}}$ is a lattice. Suppose that this lattice is distributive.
Then, in particular again by Proposition~\ref{prop:infima-suprema}
\[(\overline{f}\boxplus\overline{g})\sqcap\overline{h}\equivSW\overline{(f\boxplus g)}\sqcap\overline{h}=(f\boxplus g)\overline{\sqcap} h\leqSTW(f\overline{\sqcap}h)\boxplus(g\overline{\sqcap} h)=(\overline{f}\sqcap\overline{h})\boxplus(\overline{g}\sqcap\overline{h}),\]
i.e., $(\overline{f}\boxplus\overline{g})\sqcap\overline{h}\leqSTW(\overline{f}\sqcap\overline{h})\boxplus(\overline{g}\sqcap\overline{h})$,
which by Lemma~\ref{lem:completion-total}, Propositions~\ref{prop:infima-suprema} and \ref{prop:completion-algebraic} is equivalent to
\[(\overline{f}\boxplus\overline{g})\sqcap\overline{h}\leqSW\overline{(\overline{f}\sqcap\overline{h})\boxplus(\overline{g}\sqcap\overline{h})}
\equivSW\overline{(\overline{f}\sqcap\overline{h})}\boxplus\overline{(\overline{g}\sqcap\overline{h})}\equivSW(\overline{f}\sqcap\overline{h})\boxplus(\overline{g}\sqcap\overline{h}).\]
Hence, it suffices to provide a counterexample for $(\overline{f}\boxplus\overline{g})\sqcap\overline{h}\leqSW(\overline{f}\sqcap\overline{h})\boxplus(\overline{g}\sqcap\overline{h})$.
We use the proof idea of \cite[Theorem~4.4]{Dzh19} and we consider the constant problems $c_{p,q}:\In\IN^\IN\to\IN^\IN,p\mapsto q$ with $\dom(c_{p,q})=\{p\}$
for $p,q\in\IN^\IN$. Let $p_i,q_i\in\IN^\IN$ for $i\in\{0,1,2\}$ be mutually Turing incomparable and such
that none of these points can be computed from the supremum of the others (this is possible, see for instance~\cite[Exercise~2.2 in Chapter~VII]{Soa87}).
We choose $f:=c_{p_0,q_0}$, $g:=c_{p_1,q_1}$ and $h:=c_{p_2,q_2}$.
We recall that $\overline{\IN^\IN}=\IN^\IN\cup\{\bot\}$ is represented with a precomplete representation $\delta$,
defined by $\delta(p)=\id^\wp(p)=p-1$ for $p-1\in\IN^\IN$ and $\delta(p)=\bot$ otherwise.
Now assume that $(\overline{f}\boxplus\overline{g})\sqcap\overline{h}\leqSW(\overline{f}\sqcap\overline{h})\boxplus(\overline{g}\sqcap\overline{h})$
via computable $H,K$. We claim that $K\langle\langle i,p_i+1\rangle,p_2+1\rangle=\langle i,\langle p_i',p_2'\rangle\rangle$ for $i\in\{0,1\}$ with names $p_k'$ of $p_k$ for $k\in\{0,1,2\}$.
Firstly, if $K\langle\langle i,p_i+1\rangle,p_2+1\rangle=\langle j,\langle r,s\rangle\rangle$
such that $r$ is not a name of $p_j$ or $s$ is not a name of $p_2$, then a realizer of $e:=(\overline{f}\sqcap\overline{h})\boxplus(\overline{g}\sqcap\overline{h})$
on $\langle j,\langle r,s\rangle\rangle$ could return any value, for instance a computable one, and in this case $H$ could
neither compute $q_i$ nor $q_2$ from this result.
Hence $K\langle\langle i,p_i+1\rangle,p_2+1\rangle=\langle j,\langle p_j',p_2'\rangle\rangle$
with $j\in\{0,1\}$ and $p_k'$ a name for $p_k$ for $k\in\{0,1,2\}$.
Secondly, if $j\not=i$, then a realizer of $e$
upon input of $\langle j,\langle p_j',p_2'\rangle\rangle$ could return a name of $q_j$ together with some computable values,
from which $H$ can neither compute $q_i$ nor $q_2$. This proves the claim above.
Now on input $\langle 0,\langle p_0',p_2'\rangle\rangle$ as above, a realizer of $e$ can produce
$r:=\langle \langle 0,q_0'\rangle+1,\widehat{0}\rangle$ with a name $q_0'$ of $q_0$.
Suppose $H(r)=\langle 0,s\rangle$ with some $s\in\IN^\IN$.
Since $H$ is continuous, a certain prefix of $r$ is sufficient to produce the output $0$ in the first component.
Now on input $\langle 1,\langle p_1',p_2'\rangle\rangle$ a realizer of $e$ can produce the output
$t:=\langle\langle 0,c\rangle+1,0^n(\langle 1,q_2'\rangle+1)\rangle$ with a computable $c$ that shares
a sufficiently long prefix with $q_0'$ and a sufficiently large $n\in\IN$ and a name $q_2'$ of $q_2$. Then
$H(t)=\langle 0,s'\rangle$ with some $s'\in\IN^\IN$.
However, $s'$ is computable from $q_2$ and hence it can neither compute $q_0$ nor $q_1$,
which is a contradiction. Hence $H(r)=\langle 1,s\rangle$ with some $s\in\IN^\IN$.
Again, due to continuity of $H$, some prefix of the input is sufficient to produce the component $1$ on the output side.
On input $\langle 1,\langle p_1',p_2'\rangle\rangle$ a realizer of $e$ can now produce the output
$t:=\langle\langle 0,q_0'\rangle+1,0^n(\langle 0,q_1'\rangle+1)\rangle$ for sufficiently large $n\in\IN$
 and $H(t)=\langle 1,s'\rangle$ with $s'\in\IN^\IN$.
However, since $s'$ is computable from $q_0$ and $q_1$, it cannot compute $q_2$, which is a contradiction.
\end{proof}

We are going to prove that $+$ also co-preserves completion with respect to $\leqTW$ and $\leqSTW$.

\begin{proposition}[Sums]
\label{prop:sums}
$\overline{f+g}\equivSW\overline{f}+\overline{g}$ and hence $f+g\equivSTW\overline{f}+\overline{g}$ for all problems $f,g$.
\end{proposition}
\begin{proof}
We consider problems $f:\In X\mto Y$ and $g:\In U\mto V$. We obtain the problems  
$\overline{f}+\overline{g}:\overline{X}\times\overline{U}\mto\overline{\overline{Y}}\times\overline{\overline{V}}$ with
\[(\overline{f}+\overline{g})(x,u)=\left\{\begin{array}{ll}
   (f(x)\times\overline{\overline{V}})\cup(\overline{\overline{Y}}\times g(u)) & \mbox{if $(x,u)\in\dom(f)\times\dom(g)$}\\
   (f(x)\times\overline{\overline{V}})\cup(\overline{\overline{Y}}\times \overline{V}) & \mbox{if $x\in\dom(f)$ and $u\not\in\dom(g)$}\\
   (\overline{Y}\times\overline{\overline{V}})\cup(\overline{\overline{Y}}\times g(u)) & \mbox{if $x\not\in\dom(f)$ and $u\in\dom(g)$}\\
      (\overline{Y}\times\overline{\overline{V}})\cup(\overline{\overline{Y}}\times \overline{V}) & \mbox{otherwise}
\end{array}\right.\]
and $\overline{f+g}:\overline{X\times U}\mto\overline{\overline{Y}\times\overline{V}}$ with
\[(\overline{f+g})(z)=\left\{\begin{array}{ll}
   (f(x)\times\overline{V})\cup(\overline{Y}\times g(u)) & \mbox{if $z=(x,u)\in\dom(f)\times\dom(g)$}\\
   \overline{\overline{Y}\times\overline{V}} & \mbox{otherwise}
\end{array}\right..\]
And we also consider $h:\overline{X}\times\overline{U}\mto\overline{Y}\times\overline{V}$ with
\[h(x,u):=\left\{\begin{array}{ll}
   (f(x)\times\overline{V})\cup(\overline{Y}\times g(u)) & \mbox{if $(x,u)\in\dom(f)\times\dom(g)$}\\
   \overline{Y}\times\overline{V} & \mbox{otherwise}
\end{array}\right..\]
Then we have $h(x,u)\In(\overline{f}+\overline{g})(x,u)$ for all $(x,u)\in\overline{X}\times\overline{U}$ and hence together with 
Proposition~\ref{prop:completion-algebraic} $\overline{f+g}\leqSW\overline{f}+\overline{g}\leqSW h$.
On the other hand, there is a computable function $s:\overline{\overline{Y}\times\overline{V}}\to\overline{Y}\times\overline{V}$
with $s(y,v)=(y,v)$ for all $(y,v)\in\overline{Y}\times\overline{V}$.
Namely, one can just consider $S:\In\IN^\IN\to\IN^\IN,p\mapsto p-1$ and extend this to a
total computable realizer of $s$ under the representation of $\overline{Y}\times\overline{V}$, which is possible, since this space has a precomplete representation by Proposition~\ref{prop:products-precompleteness}. Analogously to $s$, there is also a computable function $\iota:\overline{X}\times\overline{U}\to\overline{X\times U}$
with $\iota(x,u)=(x,u)$ for $(x,u)\in X\times U$. Then $h=s\circ\overline{(f+g)}\circ\iota$ and hence
$h\leqSW\overline{f+g}$.
\end{proof}

The following example shows that $\times$ and $\sqcap$ do not co-preserve completion with respect to $\leqW$
and that $\sqcup$ does not co-preserve completion with respect to $\leqSW$.

\begin{lemma}
\label{lem:counter-monotonicity}
There are problems $f,g:\In\IN^\IN\to\IN^\IN$ such that
\begin{enumerate}
\item $\overline{f}\times\overline{g}\nleqW\overline{f\times g}$, and hence $\overline{f}\times\overline{g}\nleqTW f\times g$,
\item $\overline{f}\sqcap\overline{g}\nleqW\overline{f\sqcap g}$, and hence $\overline{f}\sqcap\overline{g}\nleqTW f\sqcap g$,
\item $\overline{f}\sqcup\overline{g}\nleqSW\overline{f\sqcup g}$, and hence $\overline{f}\sqcup\overline{g}\nleqSTW f\sqcup g$.
\end{enumerate}
\end{lemma}
\begin{proof}
We consider the constant problems $c_{p,q}:\In\IN^\IN\to\IN^\IN,p\mapsto q$ with $\dom(c_{p,q})=\{p\}$
for $p,q\in\IN$. Let $p,q,r,s\in\IN^\IN$ be mutually Turing incomparable. We choose $f:=c_{p,q}$ and $g:=c_{r,s}$.
We recall that $\overline{\IN^\IN}=\IN^\IN\cup\{\bot\}$ is represented with a precomplete representation $\delta$,
defined by $\delta(p)=\id^\wp(p)=p-1$ for $p-1\in\IN^\IN$ and $\delta(p)=\bot$ otherwise.
We only need to prove the former statements regarding $\leqW$, since the latter statements regarding $\leqTW$ follow in each case with Lemma~\ref{lem:completion-total}.\\
(1) holds since a name for the input pair $(p,p)\in\dom(\overline{c_{p,q}}\times\overline{c_{r,s}})$
can only be mapped computably to a name of an input outside of $\dom(c_{p,q}\times c_{r,s})=\{(p,r)\}$ since $r$ is not computable from $p$,
and a realizer for $\overline{c_{p,q}\times c_{r,s}}$ can map such a name to any name, for instance a computable name.
From a computable name and a name for $(p,p)$ one cannot compute $q$.\\
(2) Let us assume that $\overline{c_{p,q}}\sqcap\overline{c_{r,s}}\leqW\overline{c_{p,q}\sqcap c_{r,s}}$
is witnessed by computable $H,K$. We consider the name $p+1$ of $p$ and the name
$\widehat{0}$ of $\bot$. Since $(p,\bot)\in\dom(\overline{c_{p,q}}\sqcap\overline{c_{r,s}})$, 
$K\langle p+1,\widehat{0}\rangle$ has to be defined, but it cannot be a name of a point in
$\dom(c_{p,q}\sqcap c_{r,s})=\{p\}\times\{r\}$. Let $G$ be a realizer of $\overline{c_{p,q}\sqcap c_{r,s}}$
that maps every name of a point outside of $\dom(c_{p,q}\sqcap c_{r,s})$ to $\widehat{0}$.
Then $H\langle\langle p+1,\widehat{0}\rangle,GK\langle p+1,\widehat{0}\rangle\rangle=H\langle\langle p+1,\widehat{0}\rangle,\widehat{0}\rangle=\langle 1,t\rangle$ for some $t\in\IN^\IN$,
since it cannot be equal to $\langle 0,u\rangle$ for some $u\in\IN^\IN$ because $q$ cannot be computed from $p$ and $H\langle\id,GK\rangle$ has to be a realizer of 
$\overline{c_{p,q}}\sqcap\overline{c_{r,s}}$. Due to continuity of $H$ the output $1$ in the first component
is determined already by a prefix of the input, say by $w\prefix p+1$ and $0^n\prefix\widehat{0}$.
Hence, on the names $w\widehat{0}$ and $0^n(r+1)$ of $\bot$ and $r$, respectively, the function $H$ will also produce $1$ in the first component.
Moreover $K\langle w\widehat{0},0^n(r+1)\rangle$ is also a name of a point outside of $\dom(c_{p,q}\sqcap c_{r,s})=\{p\}\times\{r\}$ and hence $GK\langle w\widehat{0},0^n(r+1)\rangle=\widehat{0}$. 
In this case we must have $H\langle\langle w\widehat{0},0^n(r+1)\rangle,GK\langle w\widehat{0},0^n(r+1)\rangle\rangle=H\langle\langle w\widehat{0},0^n(r+1)\rangle,\widehat{0}\rangle=\langle 1,t\rangle$
with a name $t$ of $s$, which is impossible, since $s$ cannot be computed from $r$.\\
(3) Let us assume that $\overline{c_{p,q}}\sqcup\overline{c_{r,s}}\leqSW\overline{c_{p,q}\sqcup c_{r,s}}$
is witnessed by computable $H,K$. Upon input of the name $\langle i,\widehat{0}\rangle$ of $(i,\bot)\in\dom(\overline{c_{p,q}}\sqcup\overline{c_{r,s}})$ with $i\in\{0,1\}$ the function $K$
cannot produce a name of a point in $\dom(c_{p,q}\sqcup c_{r,s})=\{(0,p),(1,r)\}$. There is a realizer $G$ of $\overline{f\sqcup g}$ that produces the name $\widehat{0}$ of $\bot$
on any input outside of the domain of $\dom(c_{p,q}\sqcup c_{r,s})$ and hence $HGK\langle i,\widehat{0}\rangle=\langle j,t\rangle$ for some fixed $j\in\{0,1\}$ and $t\in\IN^\IN$ and both values $i\in\{0,1\}$.
The fixed $j$ can only be correct for one of the values $i$, since we need $i=j$ for the correctness of $H,K$, which is impossible.
\end{proof}

With the help of Corollary~\ref{cor:completion-total} it follows that $\times$ and $\sqcap$ are not
monotone with respect to the total versions of Weihrauch reducibility.

\begin{corollary}
\label{cor:non-monotonicity}
$\times,\sqcap$ are neither monotone with respect to $\leqTW$ nor with respect to $\leqSTW$,
and $\sqcup$ is not monotone with respect to $\leqSTW$.
\end{corollary}

Many further algebraic properties of the Weihrauch lattice have been studied in~\cite{BP18}.
Some of these results can be transferred to the total case by Corollary~\ref{cor:partial-total}.
In some cases we can also transfer results for pointed problems, since the completion $\overline{f}$
of any problem is always pointed. For instance, the completions of the algebraic operations are
ordered in the following way, as the corresponding reductions hold more generally for pointed problems (by \cite[Proposition~5.7]{BGP18},
and that $f^*\leqW\widehat{f}$ holds for pointed $f$, is easy to see). 

\begin{corollary}[Order of operations]
\label{cor:order-operations}
For all problems $f$ and $g$ we obtain:\linebreak
$\overline{f}+\overline{g}\leqSW \overline{f}\sqcap \overline{g}\leqSW\overline{f}\boxplus\overline{g}\leqSW \overline{f}\sqcup \overline{g}\leqW \overline{f}\times \overline{g}$, $\overline{f}\boxplus \overline{g}\leqSW\overline{f}\times\overline{g}$
and ${\overline{f}\,}^*\leqW\widehat{\overline{f}}$.
\end{corollary}

Now we study the completions of parallelization $f\mapsto\widehat{\overline{f}}$ and finite parallelization $f\mapsto{\overline{f}\,}^*$.
In \cite[Proposition~4.2]{BG11} we proved that $f\mapsto\widehat{f}$ is a closure operator for $\leqW$ and $\leqSW$
and Pauly proved in \cite[Theorem~6.2]{Pau10a} that $f\mapsto f^*$ is a closure operator for (the topological version of) $\leqW$.
We note that the latter one is not a closure operator for $\leqSW$. Nevertheless, the completions of both operators are closure operators for $\leqTW$ and $\leqSTW$.
In order to prove this, we need the following additional lemma.

\begin{lemma}[Arno Pauly\footnote{By personal communication 2018.}]
\label{lem:Pauly}
$f^{**}\equivSW f^{*}$ for all pointed problems $f$.
\end{lemma}
\begin{proof}
It is easy to see that $f\leqSW f^*$ holds for all problems $f$, in particular, we obtain $f^*\leqSW f^{**}$.
For the inverse reduction we assume that $f$ is pointed.
Let $p_0$ be a computable name of a point in $\dom(f)$. We use the computable functions $K,H$ with
\begin{eqnarray*}
&&K\langle n,\langle\langle i_1,\langle p_{1,1},...,p_{1,i_1}\rangle\rangle,\langle i_2,\langle p_{2,1},...,p_{2,i_2}\rangle\rangle,...,\langle i_n,\langle p_{n,1},...,p_{n,i_n}\rangle\rangle\rangle\rangle\\
&:=&\left\langle k,\langle p_{1,1},...,p_{1,i_1},p_{2,1},...,p_{2,i_2},......,p_{n,1},...,p_{n,i_n},\underbrace{p_0,...,p_0}_{m\mbox{ \footnotesize times}}\rangle\right\rangle
\end{eqnarray*}
where $k:=\langle n,\langle i_1,...,i_n\rangle\rangle\geq i_1+...+i_n$ and $m:=k-(i_1+...+i_n)$, and for arbitrary $k=\langle n,\langle i_1,...,i_n\rangle\rangle\in\IN$ 
and $j:=i_1+...+i_n\leq k$ we define
\begin{eqnarray*}
&&H\langle k,\langle q_1,...,q_k\rangle\rangle\\
&:=&\langle n,\langle\langle i_1,\langle q_{1},...,q_{i_1}\rangle\rangle,\langle i_2,\langle q_{i_1+1},...,q_{i_1+i_2}\rangle\rangle,...,\langle i_n,\langle q_{i_1+...+i_{n-1}+1},...,q_{j}\rangle\rangle\rangle\rangle.
\end{eqnarray*}
Then $H,K$ are computable and witness $f^{**}\leqSW f^*$.
\end{proof}

Now we are prepared to prove the following result.

\begin{proposition}[Parallelization]
\label{prop:parallelization}
$f\mapsto\widehat{\overline{f}}$ and $f\mapsto{\overline{f}\,}^*$ are closure operators for $\leqTW$ and $\leqSTW$ (and also for $\leqW$ and $\leqSW$).
\end{proposition}
\begin{proof}
Since parallelization $f\mapsto\widehat{f}$ and completion $f\mapsto\overline{f}$ are both closure operators for $\leqSW$ and $\leqW$ 
by \cite[Proposition~4.2]{BG11} and Proposition~\ref{prop:completion-closure}, and parallelization preserves completion by 
Proposition~\ref{prop:completion-algebraic}, the claim follows from Propositions~\ref{prop:closure-operators} and \ref{prop:preservation}.
The claim for $f\mapsto{\overline{f}\,}^*$ with respect to $\leqTW$ follows analogously. 
In order to prove the claim for $\leqSTW$, we note that $f\mapsto f^*$ is a closure operator with respect to $\leqSW$
restricted to pointed problems. This follows from Corollary~\ref{cor:monotonicity}, Lemma~\ref{lem:Pauly} and since $f\leqSW f^*$ obviously holds true.
Hence, we also obtain that $f\mapsto{\overline{f}\,}^*$ is a closure operator with respect to $\leqSTW$, since all problems of the form $\overline{f}$ are pointed.
\end{proof}

With the following counterexamples we show that (finite) parallelization does not co-preserve completion.
Some of the statements can be seen as a strengthening of the first statement in Lemma~\ref{lem:counter-monotonicity}.

\begin{lemma}
\label{lem:counter-monotonicity-2}
There is a problem $f$ with $\overline{f}\times\overline{f}\nleqW\overline{\widehat{f}}$ and $\overline{f}\times\overline{f}\nleqW\overline{f^*}$. This implies
\begin{enumerate}
\item $\overline{f}\times\overline{f}\nleqW\overline{f\times f}$, and hence $\overline{f}\times\overline{f}\nleqTW f\times f$,
\item ${\overline{f}\,}^*\nleqW\overline{f^*}$, and hence ${\overline{f}\,}^*\nleqTW f^*$,
\item $\widehat{\overline{f}}\nleqW\overline{\widehat{f}}$, and hence $\widehat{\overline{f}}\nleqTW\widehat{f}$.
\end{enumerate}
\end{lemma}
\begin{proof}
We consider the function $f:\In\IN^\IN\to\IN^\IN$ with $f(p)=q$, $f(r)=s$, $\dom(f)=\{p,r\}$ and pairwise Turing incomparable $p,q,r,s\in\IN^\IN$
such that none of these is computable from the supremum of the others (this is possible, see for instance~\cite[Exercise~2.2 in Chapter~VII]{Soa87}).
We recall that $\overline{\IN^\IN}=\IN^\IN\cup\{\bot\}$ is represented with a precomplete representation $\delta$,
defined by $\delta(p)=\id^\wp(p)=p-1$ for $p-1\in\IN^\IN$ and $\delta(p)=\bot$ otherwise.
Let us assume that the reduction $\overline{f}\times\overline{f}\leqW\overline{\widehat{f}}$ holds, witnessed by computable $H,K$. 
The names $p+1,r+1$ of $p,q$ are mapped by $K$ to a name $K\langle p+1,r+1\rangle$ of a point $(q_n)_{n\in\IN}$ in $\dom(\widehat{f})$,
since a realizer of $\overline{\widehat{f}}$ can choose a computable output outside of $\dom(\widehat{f})$ and 
the result $(q,s)$ cannot be computed from $p,r$ alone. For the same reason $q_n=r$ for at least one $n\in\IN$ and hence $\pi_n K\langle p+1,r+1\rangle$
is a name $r_0$ of $r$. Due to continuity of $K$ there are prefixes $w\prefix p+1$ and $v\prefix r+1$ that are sufficient for $K$ to produce a prefix
$u\prefix r_0$ that is long enough so that it cannot be extended to a name of $p$. We can now replace $r+1$ by $t=v\widehat{0}$, which is a name
of $\bot\in\dom(\overline{f})$. Now $\pi_n K\langle p+1,t\rangle$ cannot be a name of $r$, since $r$ cannot be computed from $p$ and $t$
and it cannot be a name of $p$ either, since $u\prefix\pi_n K\langle p+1,t\rangle$. Hence $K\langle p+1,t\rangle$ is a name for a point
outside of $\dom(\widehat{f})$ and a realizer of $\overline{\widehat{f}}$ can choose a computable result $c$ on this name.
But $H\langle \langle p+1,t\rangle, c\rangle$ cannot compute $q$, which is required by the assumption. 
This proves $\overline{f}\times\overline{f}\nleqW\overline{\widehat{f}}$. 
The second statement can be proved analogously, one has to choose $w,v$ such that also the natural number component of the name of an output in $(\IN^\IN)^*$ is fixed.

All other statements that involve $\leqW$ are consequences since $\overline{f}\times\overline{f}\leqW{\overline{f}\,}^*\leqW\widehat{\overline{f}}$, $\overline{f\times f}\leqW\overline{f^*}$
and $\overline{f\times f}\leqW\overline{\widehat{f}}$. 
These reductions follow since obviously $g\times g\leqW g^*$ and $g\times g\leqW\widehat{g}$ for any problem $g$, completion is a closure operator by Proposition~\ref{prop:completion-closure},
and by Corollary~\ref{cor:order-operations},
since $\overline{f}$ is pointed.
The statements that involve $\leqTW$ follow from Lemma~\ref{lem:completion-total}.
\end{proof}

As an immediate consequence of these counterexamples we can conclude that parallelization and finite parallelization are not monotone
operations for the total variants of Weihrauch reducibility. Since $\overline{f}\leqSTW f$ holds by Corollary~\ref{cor:completion-total}, we obtain the following
conclusion using Lemma~\ref{lem:counter-monotonicity-2}.
 
\begin{corollary}
\label{cor:non-monotonicity-parallelization}
$f\mapsto\widehat{f}$ and $f\mapsto f^*$ are neither monotone with respect to $\leqTW$ nor with respect to $\leqSTW$.
\end{corollary}

Another consequence of Lemma~\ref{lem:counter-monotonicity-2} is that completion does neither
preserve idempotency nor parallelizability. We recall that a problem $f$ is called {\em idempotent}, if $f\equivW f\times f$
and it is called {\em parallelizable}, if $\widehat{f}\equivW f$. If we consider the problem $f$ from Lemma~\ref{lem:counter-monotonicity-2},
then we can take $f^*$ and $\widehat{f}$ as examples to obtain the following result.

\begin{corollary}[Idempotency and parallelizability]\
\begin{enumerate}
\item There is an idempotent problem $f$ such that $\overline{f}$ is not idempotent.
\item There is a parallelizable problem $f$ such that $\overline{f}$ is not parallelizable.
\end{enumerate}
\end{corollary}

In the next step we want to clarify the relation between $\leqTW$ and $\leqSTW$ and for this purpose
we need to study cylinders.
We recall that a problem $f$ is called {\em cylinder}
if $\id\times f\leqSW f$ holds, and $\id\times f$ is called the {\em cylindrification} of $f$~\cite{BG11}. 
It follows from \cite[Proposition~4.16]{BG19} that ``total cylinders'' are exactly the usual cylinders.

\begin{corollary}[Total cylinders]
\label{cor:total-cylinder}
$\id\times f\leqSW f\iff\id\times f\leqSTW f$ holds for all problems $f$.
\end{corollary}

It is known that $g$ is a cylinder if and only if
$f\leqW g\iff f\leqSW g$ holds for all problems $f$~\cite[Proposition~3.5, Corollary~3.6]{BG11}.
We provide a similar result for the total variant of Weihrauch reducibility.

\begin{proposition}[Cylinder]
\label{prop:cylinder}
A problem $g$ is a cylinder if and only if for every problem $f$ one has
$f\leqTW g\iff f\leqSTW g$.
\end{proposition}
\begin{proof}
Let us assume that $f\leqTW g\iff f\leqSTW g$ holds for every problem $f$.
It is clear that $\id\times g\equivW g$ and hence $\id\times g\equivTW g$ by 
Corollary~\ref{cor:partial-total}. By the assumption this implies
$\id\times g\leqSTW g$ and hence $\id\times g\leqSW g$ by Corollary~\ref{cor:total-cylinder}.
This shows that $g$ is a cylinder.

For the other direction, let us now assume that $g$ is a cylinder, i.e., 
$\id\times g\leqSW g$ and hence $\id\times g\leqSTW g$ by Corollary~\ref{cor:partial-total}.
We only need to prove that $f\leqTW g$ implies $f\leqSTW g$.
Let us assume that $f\leqTW g$ holds.
Since $f\leqSW\id\times f$, we obtain $f\leqSTW\id\times f$ by Lemma~\ref{lem:completion-total}.
Now it suffices to show
$\id\times f\leqSTW\id\times g$. But this can be done by using the construction of the proof 
of~\cite[Proposition~3.5]{BG11}. By Lemma~\ref{lem:total} it suffices to note
that if $H,K$ from the proof of~\cite[Proposition~3.5]{BG11} are total,
then also the $H',K'$ constructed in the first half of that proof are total.
\end{proof}

Hence, the relations between strong and weak versions of the reducibility
can be expressed in the same way in the partial and the total case, respectively. 

We can also say something on the interaction between cylindrification and completion.
While the completion of a cylinder $f$ is only a cylinder in the trivial case that the original
problem $f$ is already strongly complete, the cylindrification of a complete problem is always complete.

\begin{proposition}[Completion and cylindrification]
\label{prop:completion-cylindrification}
Let $f$ be a problem. Then
\begin{enumerate}
\item $\overline{f}$ is a cylinder $\iff f$ is strongly complete and a cylinder,
\item $\id\times f$ is complete $\iff f$ is complete.
\end{enumerate}
The implication ``$\Longleftarrow$'' in (2) also holds for strongly complete instead of complete.
\end{proposition}
\begin{proof}
(1) If $f\equivSW\overline{f}$ and $f$ is a cylinder, then clearly $\id\times\overline{f}\leqSW\id\times f\leqSW f\leqSW\overline{f}$ and hence $\overline{f}$ is a cylinder.
If, on the other hand, $\overline{f}$ is a cylinder, then $\id\times\overline{f}\leqSW\overline{f}$.
Hence $\id\times\overline{f}\leqSTW f$ and since $\id\times\overline{f}$ is diverse, we obtain by \cite[Proposition~4.16]{BG19} that
$\id\times\overline{f}\leqSW f$.
This implies $\id\times f\leqSW\id\times\overline{f}\leqSW f$, which means that $f$ is a cylinder and
$\overline{f}\leqSW\id\times\overline{f}\leqSW f$, which means that $f$ is strongly complete.\\
(2) If $f$ is complete, then $\overline{\id\times f}\leqW\overline{\id}\times\overline{f}\leqW\id\times f$ by Propositions~\ref{prop:completion-algebraic} and \ref{prop:complete-problems},
which means that $\id\times f$ is complete. The proof in the strong case is analogous.
If, on the other hand, $\id\times f$ is complete, then $\overline{f}\leqW\overline{\id\times f}\leqW\id\times f\leqW f$,
where the first reduction holds since $f\leqW\id\times f$ and completion is a closure operator.
\end{proof}

\section{Co-Residual Operations}
\label{sec:residual}

In this section we will discuss certain algebraic operations that are co-residual operations.
In this context we have to deal with a top element in the Weihrauch lattice.
The Weihrauch lattice has no natural top element, but we can just attach a top element $\infty$ to it.
The algebraic operations are then naturally extended to the top element, so that the lattice structure 
and the order among the operations is preserved.
We are led to the following choice of values for all problems $f$ including $\infty$ (see also the discussion in \cite{BP18}):
\begin{enumerate}
\item $f\sqcap\infty=\infty\sqcap f=f$,
\item $f\sqcup\infty=\infty\sqcup f=\infty$,
\item $f\times\infty=\infty\times f=\infty$,
\item $f+\infty=\infty+f=f$,
\item $\overline{\infty}=\widehat{\infty}=\infty^*=\infty$.
\end{enumerate}
One arguable alternative could be to choose $0\times\infty=0$, given that $0\times f\equivW 0$ for all $f\not=\infty$. 
However this seems to be less natural for our purposes.
It is consistent with our usage of the term to say that a problem $f$ is {\em pointed}, if $1\leqW f$ holds. 
According to this definition $\infty$ is pointed too.

Using our universal function $U:\In\IN^\IN\to\IN^\IN$, we can define a representation $\Phi$ of certain continuous functions by
$\Phi_q(p):=U\langle q,p\rangle$ for all $p,q\in\IN^\IN$. Then any continuous $F:\In\IN^\IN\to\IN^\IN$ has an extension
of the form $\Phi_q:\In\IN^\IN\to\IN^\IN$ and for a computable $F$ we can choose a computable $q$ (see \cite{Wei00}). 
From this representation we can derive a G\"odel numbering $\varphi$ of the computable $F:\In\IN^\IN\to\IN^\IN$, i.e.,
for every computable $F$ there is some $n\in\IN$ such that $\varphi_n:\In\IN^\IN\to\IN^\IN$ extends $F$.
We also assume that $\varphi$ satisfies suitable utm- and smn-Theorems (see \cite{Wei00} for details).
We use $\Phi$ and $\varphi$ to define the compositional product and two implications.

The compositional product $f\star g$ was originally defined in \cite{BGM12} using the property (1) stated in Fact~\ref{fact:compositional} below.
It expresses a problem that can be obtained by first applying $g$ and then $f$ with some possible intermediate computation.
A corresponding compositional implication operation $g\to f$ was introduced and studied in \cite{BP18}.
It characterizes the minimal problem $h$ such that $f\leqW g\star h$ (see Fact~\ref{fact:compositional}).
Here we phrase these operations in a type free version on Baire space (as in \cite{BGP18}).
We also introduce a multiplicative implication $g\twoheadrightarrow f$, which is supposed to capture
a problem simpler than every $h$ such that $f\leqW g\times h$ (see Proposition~\ref{prop:multiplicative}).

\begin{definition}[Compositional product and implications]
Let $f,g$ be problems. We define problems $f\star g$, $(g\to f)$, $(g\twoheadrightarrow f):\In\IN^\IN\mto\IN^\IN$ by
\begin{enumerate}
\item $(f\star g)\langle q,p\rangle:=\langle\id\times f^\r\rangle\circ\Phi_q\circ g^\r(p)$, 
\item $(g\to f)(p):=\{\langle t,q\rangle:\emptyset\not=\Phi_t\circ g^\r(q)\In f^\r(p)\}$,
\item $(g\twoheadrightarrow f)(p):=\{\langle n,k,q\rangle:\emptyset\not=\varphi_n\langle q,g^\r\circ\varphi_k(p)\rangle\In f^\r(p)\}$, 
\end{enumerate}
where we assume for (2) and (3) that $\dom(g)\not=\emptyset$ or $\dom(f)=\emptyset$. 
In the case of special constants we define:
\begin{enumerate}
\item $f\star\infty:=\infty\star f:=\infty$,
\item $(g\to 0):=(g\twoheadrightarrow  0):=0$, $(0\to f):=(0\twoheadrightarrow f):=\infty$ for $f\nequivW 0$, 
\item $(\infty\to f):=(\infty\twoheadrightarrow f):=0$, $(g\to\infty):=(g\twoheadrightarrow\infty):=\infty$ for $g\not=\infty$.
\end{enumerate}
We call $f\star g$ the {\em compositional product}, $(g\to f)$ the {\em compositional implication} and $(g\twoheadrightarrow f)$ the {\em multiplicative implication}.
\end{definition}

The definition of $(g\to 0):=(g\twoheadrightarrow  0):=0$ is consistent with what is defined in the first two items (2) and (3) above.
The domains in the first items (1)--(3) are always meant to be maximal. For instance $\dom(g\to f)=\dom(f^\r)$ if $g$ is somewhere defined.
The fact that we use G\"odel numbers $n,k\in\IN$ for $(g\twoheadrightarrow f)$ actually has some reason: the crucial properties of this implication
are computability theoretic ones (see Proposition~\ref{prop:multiplicative-deduction}) and do not relativize to a topological version in an obvious way.
However, the fact that we use G\"odel numbers makes the domain of $(g\twoheadrightarrow f)$ relatively complicated. If $g$ is somewhere defined, then
\[\dom(g\twoheadrightarrow f)=\{p\in\dom(f^\r):(\exists q\in\dom(g^\r))\;q\leqT p\}.\]
For pointed $g$ (that have a computable point in the domain) the domain is more natural and we obtain $\dom(g\twoheadrightarrow f)=\dom(f^\r)$.
The following facts were proved in \cite[Corollaries~18 and 25, Theorem~24, Proposition~31]{BP18}.

\begin{fact}[Compositional product and implication]
\label{fact:compositional}
For all problems $f$ and $g$ including $\infty$:
\begin{enumerate}
\item $f\star g\equivW\max_{\leqW}\{f_0\circ g_0:f_0\leqW f,g_0\leqW g\}$,
\item $(g\to f)\equivW\min_{\leqW}\{h:f\leqW g\star h\}$,
\item $(g\to f)\leqW h\iff f\leqW g\star h$,
\item $\star$ is monotone with respect to $\leqW$ in both components,
\item $\to$ is monotone with respect to $\leqW$ in the second component and antitone in the first component.
\end{enumerate}
\end{fact}

We note that for (3) to be correct in the case of $\dom(g)=\emptyset$ and $\dom(f)\not=\emptyset$, we actually use $(g\to f)=\infty$ and
$f\star\infty=\infty\star f=\infty$.

By $\WW$ we denote the class of Weihrauch degrees including $\infty$.  We extend all the algebraic operations to degrees in the
usual way without introducing a new notation. It is known that the underlying structure is a lattice~\cite{BP18}
and together with Fact~\ref{fact:compositional}~(3) we obtain the following conclusion.

\begin{corollary}[Weihrauch algebra]
\label{cor:Weihrauch-algebra}
$(\WW,\leqW,\sqcap,\sqcup,\star,\to,0,1,\infty)$ is a deductive Weihrauch algebra that is not commutative.
\end{corollary}

For instance $\lim\star\WKL\equivW\lim\lW\WKL\star\lim$ and hence $\star$ is clearly not commutative.

We can interpret $(f\twoheadrightarrow\infty)=(f\to\infty)$ as negation operation in the Weihrauch lattice
and we formally define negation correspondingly.

\begin{definition}[Negation]
\label{def:negation}
For every problem $f$ we define its {\em negation} $\neg f$
by $\neg f:=\infty$ for $f\not=\infty$ and $\neg \infty:=0$ (the nowhere defined problem $0:\In\IN^\IN\to\IN^\IN$).
\end{definition}

It is then obvious that our negation behaves as in Jankov logic.

\begin{corollary}[Jankov rule]
\label{cor:Jankov-rule}
$\neg\neg f\sqcap\neg f\equivW 0$ is computable.
\end{corollary}

We note that $\neg\overline{f}\equivW \neg f\leqW\overline{\neg f}$, but equivalence does not hold as we obtain $\neg\overline{\infty}=0\lW1\equivW\overline{\neg\infty}$.
Here we are in particular interested in how the compositional product and the implications interact with completion in general.
We show that $\star$ co-preserves completion with respect to $\leqSW$ and $\to$ preserves completion with respect to $\leqW$.

\begin{proposition}[Completion and compositional products and implication]
\label{prop:completion-compositional}
For all problems $f,g$ including $\infty$:
\begin{enumerate}
\item $\overline{f\star g}\leqSW\overline{f}\star\overline{g}\equivSW\overline{\overline{f}\star\overline{g}}$.
\item $(\overline{g}\to\overline{f})\leqW\overline{(\overline{g}\to\overline{f})}\leqW\overline{(g\to f)}$.
\end{enumerate}
In particular $f\star g$ is (strongly) complete, if $f$ and $g$ are so.
\end{proposition}
\begin{proof}
(1) It is routine to check the claim for the special cases where the problem $\infty$ is involved.
Otherwise, it suffices to consider $f,g:\In\IN^\IN\mto\IN^\IN$ and for such problems we have
$f\star g=\langle\id\times f\rangle\circ U\circ\langle\id\times g\rangle$.
Hence, $\overline{f\star g}=\overline{\langle\id\times f\rangle}\circ \overline{U}\circ\overline{\langle\id\times g\rangle}$. Since $\overline{U}$ is computable and $\id$ is complete, this implies by Proposition~\ref{prop:completion-algebraic} 
\[\overline{f\star g}\leqW\overline{(\id\times f)}\star\overline{(\id\times g)}\leqW(\id\times\overline{f})\star(\id\times\overline{g})\leqW\overline{f}\star\overline{g}.\]
Since every compositional product is a cylinder by \cite[Lemma~17]{BP18}, we even obtain the strong Weihrauch reduction.
The equivalence follows as in Proposition~\ref{prop:completion-algebraic}.\\
(2) Since $f\leqW g\star(g\to f)$ by Fact~\ref{fact:compositional} and completion is a closure operator, we obtain with (1)
\[\overline{f}\leqW\overline{g\star(g\to f)}\leqW\overline{g}\star\overline{(g\to f)}.\]
Hence Fact~\ref{fact:compositional} implies $(\overline{g}\to\overline{f})\leqW\overline{(g\to f)}$, which in turn implies the statement,
as completion is a closure operator. 
\end{proof}

We note that neither of the reductions in (2) are equivalences in general, as the following examples show:
\begin{enumerate}
\item $(\overline{\infty}\to\overline{\infty})\equivW0\lW1\equivW\overline{(\overline{\infty}\to\overline{\infty})}$,
\item $\overline{(\overline{0}\to\overline{1})}\equivW1\lW\infty\equivW\overline{(0\to 1)}$.
\end{enumerate}
In particular $(\overline{g}\to\overline{f})$ does not need to be complete, even though $\overline{g}$ and $\overline{f}$ are.

We now want to study the multiplicative implication $(g\twoheadrightarrow f)$ somewhat further.
We first study its monotonicity properties.

\begin{proposition}[Monotonicity of multiplicative implication]
\label{prop:monotonicity-multiplicative-implication}
Let $f_i,g_i$ be problems for $i\in\{0,1\}$ including $\infty$. If $f_0\leqW f_1$,  $g_0\leqW g_1$ and $g_0$ is pointed,
then $(g_1\twoheadrightarrow f_0)\leqW(g_0\twoheadrightarrow f_1)$.
\end{proposition}
\begin{proof}
It is routine to check that the claim holds in those cases where the implication takes the values $0$ or $\infty$.
This includes the cases where $\infty$ is among $f_i,g_i$.
We break the proof for the other cases into two manageable pieces, where we either fix $f=f_0=f_1$ or $g=g_0=g_1$. 
It suffices to consider problems $g_i,f_i,g,f:\In\IN^\IN\mto\IN^\IN$ for $i\in\{0,1\}$.\\
(1) Let $g_0\leqW g_1$ hold via computable $H,K$. We prove $(g_1\twoheadrightarrow f)\leqW(g_0\twoheadrightarrow f)$.
Let us assume that $g_0$ is pointed. This implies that $g_1$ is also pointed and we also obtain $\dom(g_0\twoheadrightarrow f)=\dom(g_1\twoheadrightarrow f)=\dom(f)$.
By the smn-Theorem there are computable functions $r,h:\IN\to\IN$ such that
\begin{itemize}
\item $\varphi_{h\langle n,k\rangle}\langle \langle p,q\rangle,t\rangle=\varphi_n\langle q,H\langle\pi_1\varphi_k(p),t\rangle\rangle$,
\item $\varphi_{r(k)}(p)=K\circ\pi_2\varphi_k(p)$
\end{itemize}
for all $n,k\in\IN$ and $p,q,t\in\IN^\IN$.
Let $p\in\dom(g_1\twoheadrightarrow f)=\dom(g_0\twoheadrightarrow f)=\dom(f)$. 
Let $\langle n,k,q\rangle\in (g_0\twoheadrightarrow f)(p)$. Then we obtain
\begin{eqnarray*}
\varphi_{h\langle n,k\rangle}\langle\langle p,q\rangle,g_1\circ\varphi_{r(k)}(p)\rangle
&=& \varphi_n\langle q,H\langle\pi_1\varphi_k(p),g_1\circ\varphi_{r(k)}(p)\rangle\rangle\\
&=& \varphi_n\langle q,H\langle\pi_1\varphi_k(p),g_1\circ K\circ\pi_2\varphi_k(p)\rangle\rangle\\
&=& \varphi_n\langle q,H\langle\id,g_1\circ K\rangle\circ\varphi_k(p)\rangle\\
&\In& \varphi_n\langle q,g_0\circ\varphi_k(p)\rangle\In f(p).
\end{eqnarray*}
This means $\langle h\langle n,k\rangle,r(k),\langle p,q\rangle\rangle\in(g_1\twoheadrightarrow f)(p)$.
Since the function $H'$ with $H'\langle p,\langle n, k, q\rangle\rangle:=\langle h\langle n,k\rangle,r(k),\langle p,q\rangle\rangle$ is computable,
we obtain the desired conclusion $(g_1\twoheadrightarrow f)\leqW(g_0\twoheadrightarrow f)$.\\
(2) 
Let now $f_0\leqW f_1$ hold via computable functions $H,K$. 
We prove that we obtain $(g\twoheadrightarrow f_0)\leqW(g\twoheadrightarrow f_1)$.
By the smn-Theorem there are computable functions $r,h:\IN\to\IN$ such that
\begin{itemize}
\item $\varphi_{h\langle n,k\rangle}\langle\langle p,q\rangle,t\rangle=H\langle p,\varphi_n\langle q,t\rangle\rangle$,
\item $\varphi_{r(k)}(p)=\varphi_k\circ K(p)$
\end{itemize}
for all $n,k\in\IN$ and $p,q,t\in\IN^\IN$.
Since $g$ is pointed, we have $\dom(g\twoheadrightarrow f_i)=\dom(f_i)$ for $i\in\{0,1\}$.
Let $p\in\dom(g\twoheadrightarrow f_0)$. Then $K(p)\in\dom(g\twoheadrightarrow f_1)$.
Let $\langle n,k,q\rangle\in(g\twoheadrightarrow f_1)K(p)$. 
This means that we have $\emptyset\not=\varphi_n\langle q,g\varphi_kK(p)\rangle\In f_1K(p)$.
Then we obtain
\[
\emptyset\not=
\varphi_{h\langle n,k\rangle}\langle\langle p,q\rangle,g\varphi_{r(k)}(p)\rangle
= H\langle p,\varphi_n\langle q,g\varphi_kK(p)\rangle\rangle
\In H\langle p,f_1K(p)\rangle\In f_0(p),
\]
i.e., $\langle h\langle n,k\rangle,r(k),\langle p,q\rangle\rangle\in(g\twoheadrightarrow f_0)(p)$.
This proves $(g\twoheadrightarrow f_0)\leqW(g\twoheadrightarrow f_1)$.
\end{proof}

The pointedness assumption is not necessary when we deal with total Weihrauch reducibility.
Hence, analogously to the proof of Proposition~\ref{prop:closure-operators}
we can obtain the following conclusion.

\begin{corollary}[Monotonicity of multiplicative implication]
\label{cor:monotonicity-multiplicative-implication}
$\overline{\twoheadrightarrow}$ is monotone in the second argument and antitone in the 
first argument with respect to $\leqTW$.
\end{corollary}

Now we would like to have an analog of Fact~\ref{fact:compositional}~(3) for $\twoheadrightarrow$.
Unfortunately, this is not possible, but we can say at least the following.

\begin{proposition}[Multiplicative implication]
\label{prop:multiplicative}
For all problems $f,g$ including $\infty$:
\begin{enumerate}
\item $f\leqW g\times h\TO (g\twoheadrightarrow f)\leqW h$,
\item $(g\twoheadrightarrow f)\leqW h\TO f\leqW g\star h$, provided that $g$ is pointed,
\item $(g\to f)\leqW(g\twoheadrightarrow f)$, provided that $g$ is pointed.
\end{enumerate}
\end{proposition}
\begin{proof}
It is routine to check that the claim holds in those cases where the implication takes the values $0$ or $\infty$.
This includes the cases where $\infty$ is among $f,g,h$.
Otherwise, it suffices to consider problems $f,g,h:\In\IN^\IN\mto\IN^\IN$.\\
(1) Let $f\leqW g\times h$ be witnessed by computable functions $H$ and $K$.
Then there are $n,k\in\IN$ with $\varphi_n\langle\langle p,r\rangle,s\rangle=H\langle p,\langle s,r\rangle\rangle$ and $\varphi_k=\pi_1K$.
We need to prove $(g\twoheadrightarrow f)\leqW h$. 
We define $K',H':\In\IN^\IN\to\IN^\IN$ by $K':=\pi_2 K$ and $H'\langle p,r\rangle:=\langle n,k,\langle p,r\rangle\rangle$ for all $p,r\in\IN^\IN$ and $n,k\in\IN$.
Given an input $p\in\dom(g\twoheadrightarrow f)$ we claim that $H'\langle p,hK'(p)\rangle\In(g\twoheadrightarrow f)(p)$,
i.e., $H',K'$ witness $(g\twoheadrightarrow f)\leqW h$:
if $\langle n,k,\langle p,r\rangle\rangle\in H'\langle p,hK'(p)\rangle$, then $r\in h\pi_2 K(p)$ and hence 
\[\varphi_n\langle \langle p,r\rangle,g\varphi_k(p)\rangle\In H\langle p,\langle g\pi_1 K(p),h\pi_2 K(p)\rangle\rangle=H\langle p,\langle g\times h\rangle\circ K(p)\rangle\In f(p).\]
This means $\langle n,k,\langle p,r\rangle\rangle\in(g\twoheadrightarrow f)(p)$, which proves the claim.\\
(2) This follows from (3) together with Fact~\ref{fact:compositional}.\\
(3) Given a $p\in\dom(g\to f)$ we can use $(g\twoheadrightarrow f)$ in order to determine a $\langle n,k,q\rangle\in(g\twoheadrightarrow f)(p)$.
Here we use that $g$ is pointed and hence $\dom(f\twoheadrightarrow g)=\dom(f\to g)$.
We can then compute a $t\in\IN^\IN$ with $\Phi_t(r)=\varphi_n\langle q,r\rangle$ for all $\langle q,r\rangle\in\dom(\varphi_n)$.
We claim that $\langle t,\varphi_k(p)\rangle\in(g\to f)(p)$:
\[\Phi_t\circ g\circ\varphi_k(p)=\varphi_n\langle q,g\circ\varphi_k(p)\rangle\In f(p).\]
This proves the claim.
\end{proof}

Again the pointedness assumptions can be removed when we deal with total Weihrauch reducibility
and the corresponding completions of operations.
In this way Proposition~\ref{prop:multiplicative} shows that we have an instance of a commutative Weihrauch algebra.
We formulate this result together with the deductive Weihrauch algebra whose existence follows from Fact~\ref{fact:compositional}~(3).

\begin{corollary}[Weihrauch algebra of total Weihrauch degrees]
\label{cor:Weihrauch-algebra-total}
The total Weih\-rauch degrees give rise to the following Weihrauch algebras:
\begin{enumerate}
\item $(\WW_{\rm tW},\leqTW,\overline{\sqcap},\sqcup,\overline{\times},\overline{\twoheadrightarrow},1,1,\infty)$
        is a commutative Weihrauch algebra.
\item $(\WW_{\rm tW},\leqTW,\overline{\sqcap},\sqcup,\overline{\star},\overline{\to},1,1,\infty)$
        is a deductive Weihrauch algebra.
\end{enumerate}
\end{corollary}

It would be desirable to have an equivalence in Proposition~\ref{prop:multiplicative}~(1) instead of just an implication, which would mean
that $\twoheadrightarrow$ is a co-residual operation of $\times$ in the same way as $\to$ is a co-residual of $\star$.
However, in \cite[Proposition~37]{BP18} it was proved that there is no such co-residual operation to $\times$.
The following result shows that $\twoheadrightarrow$ has such a co-residual property at least restricted to special problems.

\begin{proposition}[Multiplicative deduction]
\label{prop:multiplicative-deduction}
$(g\twoheadrightarrow f)\leqW h\TO f\leqW\widehat{\overline{g}}\times h$ for all problems $f,g,h$ including $\infty$, such that $g$ is pointed.
\end{proposition}
\begin{proof}
It is routine to check the claim for the special cases where the problem $\infty$ is involved.
Otherwise, it suffices to consider problems $f,g,h:\In\IN^\IN\mto\IN^\IN$. 
Let $g$ be pointed and let $(g\twoheadrightarrow f)\leqW h$ hold via computable $H,K$.
Then given a point $p\in\dom(f)=\dom(g\twoheadrightarrow f)$ any $\langle n,k,q\rangle\in H\langle p,hK(p)\rangle$
satisfies $\emptyset\not=\varphi_n\langle q,g\circ\varphi_k(p)\rangle\In f(p)$.
Since $\overline{\IN^\IN}$ has a precomplete representation $\delta_{\overline{\IN^\IN}}$, 
it follows that there is a total computable universal function $u:\IN^\IN\to\IN^\IN$ with 
$\delta_{\overline{\IN^\IN}}\circ u\langle k,p\rangle=\delta_{\overline{\IN^\IN}}(\varphi_k(p)+1)=\varphi_k(p)$ for all  $k\in\IN$ and $p\in\dom(\varphi_k)$.
We define a total computable function $K'(p):=\langle\langle u\langle 0,p\rangle,u\langle 1,p\rangle,u\langle 2,p\rangle,...\rangle,K(p)\rangle$
and a computable function $H'\langle p,\langle\langle q_0,q_1,q_2,...\rangle,r\rangle\rangle:=\varphi_n\langle q,q_k-1\rangle$ where $\langle n,k,q\rangle=H\langle p,r\rangle$.
Whenever $G$ is a realizer of $\widehat{\overline{g}}$, with respect to $\delta_{\overline{\IN^\IN}^\IN}$, then we obtain
\begin{eqnarray*}
H'\langle p,\langle G\times h\rangle\circ K'(p)\rangle
&=& H'\langle p,\langle G\langle u\langle 0,p\rangle,u\langle 1,p\rangle,u\langle 2,p\rangle,...\rangle,hK(p)\rangle\rangle\\
&\In& \bigcup\left\{\varphi_n\langle q,g\circ \varphi_k(p)\rangle:\langle n,k,q\rangle\in H\langle p,hK(p)\rangle\right\}\\
&\In& f(p),
\end{eqnarray*}
i.e., $ f\leqW\widehat{\overline{g}}\times h$.
\end{proof}

The basic idea of the proof is that using the parallelization we can evaluate $\overline{g}$ on all possible inputs $\varphi_k(p)$ with G\"odel numbers
$k\in\IN$ and only after we learn the result of $h$ we know which of these values is actually needed. The completion guarantees that all these
values actually exist. 

A similar idea as in the proof of Proposition~\ref{prop:multiplicative-deduction} has been independently used by Neumann and Pauly~\cite[Proposition~31]{NP18}
to prove the following result, which we rephrase in terms of our terminology.\footnote{The notion of precompleteness used by Neumann and Pauly is not the usual one; what
is required is rather a uniform version of completeness, which is satisfied by our completion $\overline{g}$.} 

\begin{proposition}[Neumann and Pauly 2018]
$g\star h\leqW \widehat{\overline{g}}\times h$ for all problems $g$ and $h:\In X\mto\IN$.
\end{proposition}

This result yields a similar transition from $g\star h$ to $\widehat{\overline{g}}\times h$ as the one that happens from 
Proposition~\ref{prop:multiplicative} to \ref{prop:multiplicative-deduction}, except that we do not need problems $h$ with natural number output for the latter transition.
We obtain the following obvious corollary of Proposition~\ref{prop:multiplicative-deduction}.

\begin{corollary}[Multiplicative deduction]
\label{cor:multiplicative-deduction}
$(g\twoheadrightarrow f)\leqW h\iff f\leqW g\times h$ for all problems $f,g,h$ including $\infty$  and such that $g$ is parallelizable and complete.
\end{corollary}

This is the key observation that is used in the next section in order to show that the parallelized total Weihrauch degrees form a Brouwer algebra.
We note that by \cite[Proposition~37]{BP18} it is known that there is no way to define $\twoheadrightarrow$ such that 
the statement in Corollary~\ref{cor:multiplicative-deduction} holds for all problems $g$.
This remains so, even if we replace Weihrauch reductions $\leqW$ by total Weihrauch reductions $\leqTW$ and the product $\times$ by its completion $\overline{\times}$,
as a refined version of the argument from \cite[Proposition~37]{BP18} shows.

\begin{proposition}
The operation $\overline{\times}$ is not co-residuated and $\overline{\star}$ is not left co-residuated with respect to $\leqTW$.
\end{proposition}
\begin{proof}
We have 
\begin{enumerate}
\item $\C_{2^\IN}\times\overline{\C_\IN}\leqW \overline{\C_\IN}\times(\C_{2^\IN}\sqcup\overline{\C_\IN})$,
\item $\C_{2^\IN}\times\overline{\C_\IN}\leqW \C_{2^\IN}\times(\C_{2^\IN}\sqcup\overline{\C_\IN})$,
\item $(\C_{2^\IN}\sqcap\overline{\C_\IN})\star(\C_{2^\IN}\sqcup\overline{\C_\IN})\leqW\C_{2^\IN}\sqcup(\overline{\C_{\IN}}\star\overline{\C_{\IN}})$,
\item $\C_{2^\IN}\times\overline{\C_\IN}\nleqW\C_{2^\IN}\sqcup(\overline{\C_{\IN}}\star\overline{\C_{\IN}})$.
\end{enumerate}
While (1) and (2) are clear, it remains to justify (3) and (4). We obtain (3) since $\C_{2^\IN}\star\C_{2^\IN}\equivW\C_{2^\IN}$ and by distributivity properties of $\star$ \cite[Proposition~39]{BP18}
\begin{eqnarray*}
(\C_{2^\IN}\sqcap\overline{\C_\IN})\star(\C_{2^\IN}\sqcup\overline{\C_\IN})
&\equivW& ((\C_{2^\IN}\sqcap\overline{\C_\IN})\star\C_{2^\IN})\sqcup((\C_{2^\IN}\sqcap\overline{\C_\IN})\star\overline{\C_{\IN}})\\
&\leqW& (\C_{2^\IN}\sqcap(\overline{\C_\IN}\star\C_{2^\IN}))\sqcup((\C_{2^\IN}\star\overline{\C_\IN})\sqcap(\overline{\C_{\IN}}\star\overline{\C_{\IN}}))\\
&\leqW& \C_{2^\IN}\sqcup(\overline{\C_{\IN}}\star\overline{\C_{\IN}}).
\end{eqnarray*}
Now we need to justify why (4) holds. Since $\C_{2^\IN}$ is a fractal by \cite[Corollary~5.6]{BBP12}, \cite[Fact~3.2]{BGM12} 
and $\overline{\C_\IN}$ is a fractal as proved in \cite[Lemma~8.7]{BG19},
it follows that $\C_{2^\IN}\times\overline{\C_\IN}$ is a fractal and hence join irreducible by \cite[Proposition~2.6]{BGM12} .
This means that $\C_{2^\IN}\times\overline{\C_\IN}\leqW\C_{2^\IN}\sqcup(\overline{\C_{\IN}}\star\overline{\C_{\IN}})$
would imply that $\C_{2^\IN}\times\overline{\C_\IN}\leqW\C_{2^\IN}$ or $\C_{2^\IN}\times\overline{\C_\IN}\leqW\overline{\C_{\IN}}\star\overline{\C_{\IN}}$ holds.
The latter is impossible, as $\C_{2^\IN}$ has computable inputs without computable solutions, while $\overline{\C_\IN}\star\overline{\C_\IN}$ has
computable solutions for all inputs. The former is impossible as even $\C_\IN\nleqW\C_{2^\IN}$. 

By Propositions~\ref{prop:completion-algebraic} and \ref{prop:completion-compositional} all degrees that appear in (1)--(4) are complete,
as $\C_{2^\IN}$ is complete by Corollary~\ref{cor:complete-problems}.
Hence, all the statements (1)--(4) hold true if we replace $\leqW$ by $\leqTW$.
Suppose now a binary operation $\Box$ would exist such that 
$(g\Box f)\leqTW h\iff f\leqTW h\overline{\times} g$ holds for all problems $f,g,h$.
We consider $g:=(\C_{2^\IN}\sqcup\overline{\C_\IN})$ and $h_1:=\overline{\C_\IN}$, $h_2:=\C_{2^\IN}$
and $f:=\C_{2^\IN}\times\overline{\C_\IN}$.
Then by (1)--(4) $f\leqTW h_1\overline{\times} g$ and $f\leqTW h_2\overline{\times} g$, but $f\nleqTW (h_1\overline{\sqcap} h_2)\overline{\star} g$,
which also implies $f\nleqTW (h_1\overline{\sqcap} h_2)\overline{\times} g$
This simultaneously shows that $\Box$ does not exist and also a corresponding operation for $\overline{\star}$ does not exist.
\end{proof}

The Weihrauch algebra of total Weihrauch degrees fails in two different ways being a model of some
intuitionistic linear logic. The multiplicative and compositional versions of the algebra both fail to be Troelstra
algebras, the former is not deductive, the latter is not commutative.

\begin{corollary}
The Weihrauch algebras from Corollary~\ref{cor:Weihrauch-algebra-total} are not Troelstra algebras, i.e., 
\begin{enumerate}
\item $(\WW_{\rm tW},\leqTW,\overline{\sqcap},\sqcup,\overline{\times},\overline{\twoheadrightarrow},1,1,\infty)$
        is not deductive, 
\item $(\WW_{\rm tW},\leqTW,\overline{\sqcap},\sqcup,\overline{\star},\overline{\to},1,1,\infty)$
        is not commutative.
\end{enumerate}
\end{corollary}

\section{The Brouwer Algebra of Parallelizable Total Degrees}
\label{sec:Brouwer}

In \cite{BG11} we have already studied parallelized Weihrauch reducibility $\leq_{\rm pW}$, which is the reducibility
that is generated by the closure operator of parallelization on $\leqW$.
Likewise we want to study parallelized total Weihrauch reducibility $\leq_{\rm ptW}$.

\begin{definition}[Parallelized Weihrauch reducibility]
For problems $f,g$ we write
\begin{enumerate}
\item $f\leq_{\rm pW} g:\iff f\leqW\widehat{g}$ \hfill (parallelized Weihrauch reducibility)
\item $f\leq_{\rm ptW} g:\iff f\leqW\widehat{\overline{g}}$ \hfill (parallelized total Weihrauch reducibility)
\end{enumerate}
Analogously, we write $\equiv_{\rm pW}$ and $\equiv_{\rm ptW}$ for the corresponding equivalences.
\end{definition}

It is clear that $\leq_{\rm pW}$ and $\leq_{\rm ptW}$ are actually preorders by Propositions~\ref{prop:closure-operators},
as completion and parallelized completion are closure operators (the latter by Proposition~\ref{prop:parallelization}).
We note that we also have $f\leq_{\rm ptW} g\iff f\leqTW\widehat{\overline{g}}$ by Proposition~\ref{prop:completion-algebraic}.
It is important to mention that the order in which we apply the closure operators matters. While $\widehat{\overline{g}}$ is always complete and parallelizable,
$\overline{\widehat{g}}$ is always complete, but not necessarily parallelizable (see Lemma~\ref{lem:counter-monotonicity-2}).

For each operation $\Box\in\{\times,\sqcup,\boxplus,\sqcap,+,\star,\to,\twoheadrightarrow\}$ 
we define its {\em parallelized completion} $\widehat{\overline{\Box}}$ by $f\widehat{\overline{\Box}} g:=\widehat{\overline{f}}\Box\widehat{\overline{g}}$.
Since parallelized completion is a closure operator for $\leqW$ by Proposition~\ref{prop:parallelization},
we straightforwardly obtain the following by Proposition~\ref{prop:closure-operators}.

\begin{corollary}[Monotonicity]
\label{cor:monotonicity-parallelized}
\
\begin{enumerate}
\item
$(f,g)\mapsto f\widehat{\overline{\Box}}g$ for $\Box\in\{\times,\sqcup,\boxplus,\sqcap,+,\star\}$
is monotone with respect to $\leq_{\rm ptW}$. 
\item $(f,g)\mapsto f\widehat{\overline{\Box}}g$ for $\Box\in\{\to,\twoheadrightarrow\}$ is
monotone with respect to $\leq_{\rm ptW}$ in the second argument and antitone in the first argument.
\end{enumerate}
\end{corollary}
\begin{proof}
The corresponding monotonicity properties with respect to $\leqW$ are known by \cite[Proposition~3.6]{BGP18}, except for $\twoheadrightarrow$:
$\overline{\twoheadrightarrow}$ is monotone with respect to $\leqTW$ by Corollary~\ref{cor:monotonicity-multiplicative-implication}.
Hence the claims follow from Proposition~\ref{prop:closure-operators}.
\end{proof}

An interesting property of parallelized (total) Weihrauch reducibility is that suprema and products are merged in a certain sense.
We summarize some facts regarding preservation and co-preservation of parallelization 
that were proved in \cite[Propositions~4.5, 4.8, 4.9]{BG11} and \cite[Propositions~41, 44]{BP18}.

\begin{fact}[Parallelization and algebraic operations]
\label{fact:parallelization}
For all problems $f,g$ including $\infty$:
\begin{enumerate}
\item $\widehat{f\times g}\equivSW\widehat{f}\times\widehat{g}\leqW\widehat{f\sqcup g}$,
\item $\widehat{f}\sqcup\widehat{g}\leqW\widehat{f\sqcup g}\leqW\widehat{\widehat{f}\sqcup\widehat{g}}$,
\item $\widehat{f\sqcap g}\leqSW\widehat{f}\sqcap\widehat{g}\equivSW\widehat{\widehat{f}\sqcap\widehat{g}}$,
\item $\widehat{f\star g}\leqSW\widehat{f}\star\widehat{g}\equivSW\widehat{\widehat{f}\star\widehat{g}}$,
\item $\widehat{\overline{f}\times \overline{g}}\equivSW\widehat{\overline{f}}\times\widehat{\overline{g}}\equivW\widehat{\overline{f}\sqcup \overline{g}}\equivW\widehat{\overline{f\sqcup g}}$.
\end{enumerate}
\end{fact}

Hence, $\overline{\times}$ and $\sqcup$ are equivalent operations under parallelized total Weihrauch reducibility.
This follows from Fact~\ref{fact:parallelization} and Proposition~\ref{prop:preservation-suprema-infima}.

\begin{corollary}[Products and coproducts]
\label{cor:products-coproducts-parallelized}
$\overline{f}\times\overline{g}\equiv_{\rm ptW}\widehat{\overline{f}}\times\widehat{\overline{g}}\equiv_{\rm ptW}\widehat{\overline{f}}\sqcup\widehat{\overline{g}}\equiv_{\rm ptW}f\sqcup g$
for all problems $f,g$.
\end{corollary}

By $\WW_{\rm ptW}$ we denote the class of parallelized total Weihrauch degrees including $\infty$. 
We use the same notation $\leq_{\rm ptW}$ for the order on degrees and we consider the operations to be extended to these degrees.
In order to avoid too clumsy notation we use the abbreviation $\Rrightarrow$ for $\widehat{\overline{\twoheadrightarrow}}$ in the following.
We prove that the parallelized total Weihrauch degrees form a Brouwer algebra.

\begin{theorem}[Brouwer algebra]
\label{thm:Brouwer-algebra}
$(\WW_{\rm ptW},\leq_{\rm ptW},\widehat{\overline{\sqcap}},\sqcup,\Rrightarrow,1,\infty)$ is a Brouwer algebra.
\end{theorem}
\begin{proof}
$(\WW_{\rm ptW},\leq_{\rm ptW},\widehat{\overline{\sqcap}},\sqcup)$ is a lattice by Proposition~\ref{prop:closure-operators} as
parallelized completion is a closure operator.
We obtain by Corollary~\ref{cor:multiplicative-deduction} and Fact~\ref{fact:parallelization}
\begin{eqnarray*}
(g\Rrightarrow f)\leq_{\rm ptW} h
&\iff& (\widehat{\overline{g}}\twoheadrightarrow\widehat{\overline{f}})\leqW\widehat{\overline{h}}\\
&\iff& \widehat{\overline{f}}\leqW\widehat{\overline{g}}\times\widehat{\overline{h}}\\
&\iff& \widehat{\overline{f}}\leqW\widehat{\overline{g\sqcup h}}\\
&\iff& f\leq_{\rm ptW}g\sqcup h.
\end{eqnarray*}
This proves the claim.
\end{proof}

In \cite{BG11} we have proved that the Medvedev lattice can be embedded into the parallelized Weihrauch lattice.
This embedding can actually be extended to a Brouwer algebra embedding into the parallelized total Weihrauch lattice.
We recall some basic definitions for the Medvedev lattice \cite{Sor96}. Let $A,B\In\IN^\IN$. Then $A$ is said to be 
{\em Medvedev reducible} to $B$, in symbols $A\leqM B$, if there is a computable function $F:\In\IN^\IN\to\IN^\IN$
such that $B\In\dom(F)$ and $F(B)\In A$. We recall the definition of the algebraic operations of the Medvedev lattice:
\begin{enumerate}
\item $A\otimes B:=0A\cup 1B=\langle A\sqcup B\rangle$,
\item $A\oplus B:=\langle A\times B\rangle$,
\item $B\to A:=\{\langle n,q\rangle\in\IN^\IN:(\forall p\in B)\;\varphi_n\langle q,p\rangle\in A\}$.
\end{enumerate}
By $\MM$ we denote the set of Medvedev degrees. We identify degrees with their members and use the same notation for the algebraic operations on degrees.
Medvedev~\cite{Med55} proved that $(\MM,\otimes,\oplus,\to,\IN^\IN,\emptyset)$ is a Brouwer algebra (see \cite[Theorem~9.1]{Sor96}).
In \cite{BG11} we have considered the constant problems
\[c_A:\IN^\IN\mto\IN^\IN,p\mapsto A\]
for every non-empty $A\In\IN^\IN$ and $c_\emptyset=\infty$.
The following facts were proved in \cite[Theorem~5.1]{BG11}.

\begin{fact}[Medvedev embedding]
\label{fact:Medvedev}
For all $A,B\In\IN^\IN$:
\begin{enumerate}
\item $A\leqM B\iff c_A\leqW c_B$,
\item $c_{A\otimes B}\equivSW c_A\sqcap c_B$,
\item $c_{A\oplus B}\equivSW c_A\times c_B\equiv c_A\star c_B$.
\end{enumerate}
\end{fact}

The equivalence $c_A\times c_B\equivW c_A\star c_B$, was not proved in the references, but it is easy to see.
For one, $f\times g\leqW f\star g$ holds in general and on the other hand, $c_A\star c_B=\langle\id\times c_A\rangle\circ U\circ\langle\id\times c_B\rangle\leqW c_A\times c_B$,
as the output of $c_A$ does not depend on the input.
Here we add the observation that also the implication is preserved.
In fact, since the product and the compositional product for problems of the form $c_A$ coincide,
also the multiplicative and compositional implications coincide.

\begin{lemma}[Medvedev implication]
\label{lem:Medvedev-implication}
$c_{B\to A}\equivW(c_B\twoheadrightarrow c_A)\equivW(c_B\to c_A)$ for all $A,B\In\IN^\IN$.
\end{lemma}
\begin{proof}
It is routine to check the special cases of problems that involve $A,B\in\{\emptyset,\IN^\IN\}$.
Since the Medvedev lattice is a Brouwer algebra by \cite[Theorem~9.1]{Sor96}, we have $A\leqM B\oplus(B\to A)$.
With the help of Proposition~\ref{prop:multiplicative} and Fact~\ref{fact:Medvedev} we obtain
\begin{eqnarray*}
A\leqM B\oplus(B\to A)
&\TO& c_A\leqW c_{B\oplus(B\to A)}\equivW c_B\times c_{B\to A}\\
&\TO& (c_B\twoheadrightarrow c_A)\leqW c_{B\to A}.
\end{eqnarray*}
We can also prove $c_{B\to A}\leqW (c_B\twoheadrightarrow c_A)$.
Given a $p\in\IN^\IN$ we obtain
\begin{eqnarray*}
\langle n,k,q\rangle\in (c_B\twoheadrightarrow c_A)(p)
&\iff& \emptyset\not=\varphi_n\langle q,c_B\circ\varphi_k(p)\rangle\In c_A(p)\\
&\iff& \langle n,q\rangle\in c_{B\to A}(p).
\end{eqnarray*}
Hence, $c_{B\to A}\leqW (c_B\twoheadrightarrow c_A)$ follows.
We have $(c_B\to c_A)\leqW(c_B\twoheadrightarrow c_A)$ by Proposition~\ref{prop:multiplicative}. 
We also obtain $(c_B\twoheadrightarrow c_A)\leqW(c_B\to c_A)$. To this end, let $h$ be a problem
such that $c_A\leqW c_B\star h$. Like above we obtain $c_B\star h\leqW c_B\times h$, since the output of $c_B$ does not depend on its input.
That means $c_A\leqW c_B\times h$ and hence $(c_B\twoheadrightarrow c_A)\leqW h$ by Proposition~\ref{prop:multiplicative}. 
However, if $g$ is a problem such that $c_A\leqW c_B\star h$ implies $g\leqW h$ for every $h$, then $g\leqW(c_B\to c_A)$ follows.
Hence, $(c_B\twoheadrightarrow c_A)\leqW(c_B\to c_A)$.
\end{proof}

Hence the map $A\mapsto c_A$ is a lattice embedding from the Medvedev lattice into the Weihrauch
lattice that also preserves the corresponding implications (even though the Weihrauch lattice itself is not a Brouwer algebra).
It is easy to see that every Weihrauch degree of the form $c_A$ with $A\In\IN^\IN$ is parallelizable and complete,
i.e., $\widehat{\overline{c_A}}\equivW c_A$. Hence the above embedding is also an embedding into the 
parallelized total Weihrauch degrees. We note that $c_\emptyset=\infty$ and $c_{\IN^\IN}\equivW 1$.
Hence, we obtain a {\em Brouwer algebra embedding}, i.e., a lattice embedding that preserves the implication and the 
lower and upper bound.

\begin{theorem}[Embedding of the Medvedev lattice]
\label{thm:Medvedev}
$c:\MM\to\WW_{\rm ptW},A\mapsto c_A$
is a Brouwer algebra embedding.
\end{theorem}

The fact that the parallelized total Weihrauch lattice is a Brouwer algebra implies that it is a model for some intermediate logic
(i.e., some propositional logic intermediate between intuitionistic logic and classical logic). 
The existence of an embedding from the Medvedev lattice into the parallelized total Weihrauch lattice allows us to conclude 
that the logic of the parallelized total Weihrauch lattice is {\em Jankov logic},
i.e., the deductive closure of intuitionistic logic together with the {\em weak principle of excluded middle} $\neg\neg A\vee\neg A$.
We follow Sorbi~\cite{Sor91a,Sor96} for a formal definition of the theory of a Brouwer algebra.
Let $\Form$ denote the set of well formed propositional formulas. Then we call a map
$v:\Form\to\WW_{\rm ptW}$  {\em valuation} if it satisfies the following for all $A,B\in\Form$:
\begin{enumerate}
\item $v(A\vee B)=v(A) \widehat{\overline{\sqcap}} v(B)$,
\item $v(A\wedge B)=v(A)\sqcup v(B)$,
\item $v(A\to B)=(v(A)\Rrightarrow v(B))$,
\item $v(\neg A)=(v(A)\Rrightarrow\infty)$.
\end{enumerate}

We write $\WW_{\rm ptW}\vDash A$ if $v(A)=1$ for all valuations $v$. 
Then the set of formulas ${\rm Th}(\WW_{\rm ptW}):=\{A\in\Form:\WW_{\rm ptW}\vDash A\}$ is called 
the {\em theory} of $\WW_{\rm ptW}$.
It was proved by Medvedev~\cite{Med62} (see \cite[Corollary~6.4]{Sor96}) that the theory of the Brouwer algebra $\MM$ is Jankov logic.
We obtain the same result for our Brouwer algebra $\WW_{\rm ptW}$.
For one, it contains Jankov logic by Corollary~\ref{cor:Jankov-rule}. 
On the other hand, it cannot validate any additional propositional formulas as the Medvedev Brouwer algebra
is embeddable by Theorem~\ref{thm:Medvedev}.

\begin{corollary}[Theory of the parallelized complete Weihrauch degrees]
The theory of the Brouwer algebra $\WW_{\rm ptW}$ is Jankov logic.
\end{corollary}

We note that Higuchi and Pauly proved \cite[Theorems~4.1, 4.2]{HP13} that neither the Weihrauch lattice
by itself nor the parallelized Weihrauch lattice (restricted to the pointed problems) is a Brouwer algebra.
Hence, the closure operator of completion seems to be essential in order to obtain a Brouwer algebra.

In view of Corollary~\ref{cor:Weihrauch-algebra-total} one could obtain a way to transform the total Weihrauch 
lattice into a Troelstra algebra by restricting it to a linear fragment. 
We call $\LL\In\WW_{\rm ptW}$ {\em linear} if $f\times g\equivTW f*g$ holds for all $f,g\in\LL$.
If there would be any linear sublattice of interest that also preserves the
monoid structure, then that would be a potential candidate for a Troelstra algebra. 
We note that the constant multi-valued problems $c_A$ used for the embedding of the Medvedev lattice
form a linear subset of the total Weihrauch degrees by Fact~\ref{fact:Medvedev}, however, this is not a sublattice
and leads directly to a Brouwer algebra, i.e., a trivial example of a Troelstra algebra.

\begin{figure}[htb]
\begin{center}
\begin{tikzpicture}[scale=0.8,every node/.style={fill=black!15}]

\node (WKL'') at (1.6,16.8) {$\WKL''\equiv_{\rm ptW}\RT{2}{k+2}\equiv_{\rm ptW}\LLPO''$};
\node (lim') at (1.6,15.6) {$\lim'\equiv_{\rm ptW}\LPO'$};
\node (WKL') at (1.6,14.4) {$\WKL'\equiv_{\rm ptW}\KL\equiv_{\rm ptW}\BWT_\IR\equiv_{\rm ptW}\RT{1}{k+2}\equiv_{\rm ptW}\LLPO'$};
\node (lim) at (1.6,13.2) {$\lim\equiv_{\rm ptW}\SORT\equiv_{\rm ptW}\C_\IR\equiv_{\rm ptW}\C_\IN\equiv_{\rm ptW}\LPO$};
\node (DNC3) at (-1.8,11) {$\DNC_3$};
\node (DNC4) at (-1.8,10) {$\DNC_{4}$};
\node (DNCN) at (-1.8,9) {$\DNC_\IN$};
\node (PA) at (0.6,9) {$\PA$};
\node (MLR) at (1.6,10) {$\MLR$};
\node (COH) at (5.8,12) {$\COH$};
\node (1-GEN) at (8.8,12) {$1\dash\GEN$};
\node (NON) at (1.6,8.2) {$\NON$};
\node (WKL) at (-1.8,12) {$\WKL\equiv_{\rm ptW}\C_{2^\IN}\equiv_{\rm ptW}\WWKL\equiv_{\rm ptW}\IVT\equiv_{\rm ptW}\K_\IN\equiv_{\rm ptW}\LLPO$};

\draw [->] (WKL'') edge (lim');
\draw [->] (lim') edge (WKL');
\draw [->] (WKL') edge (lim);
\draw [->] (lim) edge (WKL);
\draw [->] (lim) edge (COH);
\draw [->] (lim) edge (1-GEN);
\draw [->] (WKL) edge (DNC3);
\draw [->] (DNC3) edge (DNC4);
\draw [->] (DNC4) edge (DNCN);
\draw [->] (DNC4) edge (PA);
\draw [->] (WKL) edge (MLR);
\draw [->] (DNCN) edge (NON);
\draw [->] (PA) edge (NON);
\draw [->] (MLR) edge (NON);
\draw [->] (COH) edge (NON);
\draw [->] (1-GEN) edge (NON);

\end{tikzpicture}
\end{center}
\ \\[-0.3cm]
\caption{Problems in the parallelized total Weihrauch lattice $\WW_{\rm ptW}$}
\label{fig:WptW}
\end{figure}
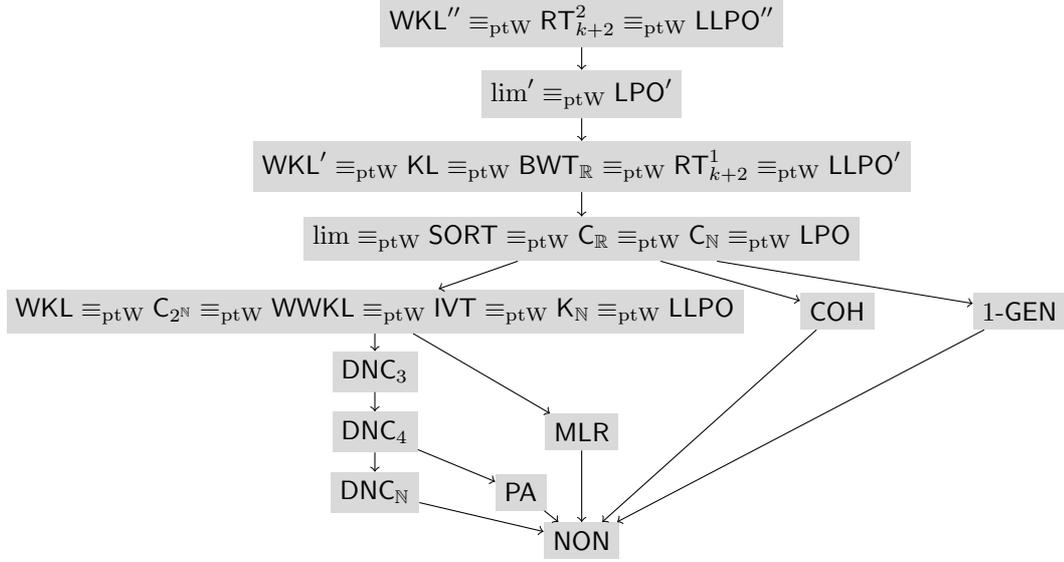

\section{Conclusion}
\label{sec:conclusion}

We have proved that the Weihrauch lattice can be transformed into a Brouwer algebra
by completion followed by parallelization. It would be desirable to understand
the structure of this Brouwer algebra somewhat better. Is it isomorphic to the
Medvedev Brouwer algebra? Presumably not, as the Medvedev algebra considers
only problems that are independent of the input. However, we need more structural
information on the lattices and algebras in order to prove such properties. 
The Medvedev lattice has, for instance, a second smallest degree, called $0'$,
which consists of all non-computable $p\in\IN^\IN$. Is there such a 
second smallest degree in the parallelized total Weihrauch lattice?
Or is the structure dense? We do not even know the answer to this question for the ordinary
Weihrauch lattice or its total variant.
What we can say, though, is that the parallelized total Weihrauch lattice 
is still inhabited by a variety of interesting problems.
The diagram in Figure~\ref{fig:WptW} shows a number of problems (that are taken without
further explanation from \cite{BHK17a} and \cite{BR17}), and that inhabit
$\WW_{\rm ptW}$. Even though a lot of problems that are normally separated
in the Weihrauch lattice are identified in $\WW_{\rm ptW}$, the structure is still rich 
and non-linear.

\section*{Acknowledgments}

We would like to thank Paulo Oliva for discussions of models of intuitionistic linear logic  at the Logic Colloquium 2018 in Udine
that have helped us to identify the relevance of Troelstra and Weihrauch algebras.

\bibliographystyle{plain}
\bibliography{C:/Users/Vasco/Dropbox/Bibliography/lit}
%\bibliography{C:/Users/vbrattka/Dropbox/Bibliography/lit}
%\bibliography{C:/Users/i11avabr/Dropbox/Bibliography/lit}

\end{document}